\documentclass[leqno,11pt]{amsart}
\usepackage{comment}

\usepackage{amsmath}
\usepackage{graphicx, color}
\usepackage{amscd}
\usepackage{amsfonts}
\usepackage{amssymb}
\usepackage{mathrsfs}
\usepackage{mathtools}
\usepackage{ulem}

\newcommand{\bel}{\begin{equation} \label}

\usepackage{hyperref}
\newtheorem{thm}{Theorem}[section]

\newtheorem{lem}[thm]{Lemma}
\newtheorem{prop}[thm]{Proposition}

\newtheorem{rem}{Remark}[section]

\numberwithin{equation}{section}

\newcommand{\be}{\begin{equation} \label}
\newcommand{\ee}{\end{equation}}
\newcommand\R{\mathbb R}

\newcommand\N{\mathbb N}

\newcommand\ts{\textstyle}

\allowdisplaybreaks
\def\eps{\varepsilon}

\newcommand\rn{\R^n}

\newcommand\ld{\mathcal{L}}

\newcommand\ldu{\ld u}

\newcommand{\myc}[1]{}

\title[Sharp gradient bounds for elliptic PDE]
{Sharp gradient bounds for uniformly elliptic PDE, \\
 eigenvalue asymptotics, and the Landis conjecture}

\author[Sirakov]{Boyan Sirakov}
\address{PUC-Rio, Departamento de Matematica \\
Rua Marqu\^es de S\~ao Vicente 225 \\
G\'avea, Rio de Janeiro - CEP 22451-900, Brazil}
\email{bsirakov@puc-rio.br}

\author[Souplet]{Philippe Souplet}
\address{Universit\'e Sorbonne Paris Nord,
CNRS UMR 7539, Laboratoire Analyse, G\'{e}om\'{e}trie et Applications,
93430 Villetaneuse, France}
\email{souplet@math.univ-paris13.fr}

\begin{document}

\begin{abstract}  We study three classical problems in the theory of linear second-order uniformly elliptic equations: (i)  gradient estimates for the Dirichlet problem, (ii) estimates and  asymptotics for the first eigenvalue of an elliptic operator, and (iii) the Landis conjecture on exponential decay. 

Our main results give versions of (i) in which the constant is sharply specified in terms of  the norms of the coefficients of the operator and the size of the domain; and use these
to  strongly improve on known results for (ii) and (iii)
by allowing both more general operators and weaker regularity assumptions on the coefficients,
 and by giving  quantitative estimates. 

The proofs use a unified approach, relying on three ingredients:
 interior and boundary Harnack inequalities with sharp constants, 
 duality arguments, and a $C^1$ estimate based on a rescaling procedure. The sharpness of the gradient and spectral estimates is demonstrated through various counterexamples.
\end{abstract}

\maketitle

\section{Introduction}
We consider uniformly elliptic equations  in general divergence form
\begin{equation}\label{defdiv}
\ldu:=\mathrm{div}(A(x) \nabla u +  b_1(x)u) + b_2(x)\cdot \nabla u +c(x) u =f(x) ,
\end{equation}
in a bounded domain $\Omega\subset\rn$ ($n\ge 2$ unless otherwise specified) with $C^{1,\bar\alpha}$-boundary
for some $\bar\alpha\in (0,1]$, and study weak (sub-, super-)solutions $u\in H^1(\Omega)$   of
\eqref{defdiv}. In most of our results we make the following standard assumptions on the coefficients:
\begin{eqnarray}
&\hskip 1cm\hbox{$A(x)$ is a  matrix such that $\Lambda I \ge A\ge\lambda I$ in $\Omega$, for some $\Lambda\ge \lambda>0$,} \label{hyp1} \\
\noalign{\vskip 1mm}
&\hskip 1cm\hbox{$A, b_1\in C^{\alpha}(\Omega)$} , \quad\hbox{$ b_2, c,f \in L^q(\Omega)$, for some $q\in(n,\infty], \ \alpha\in(0,\min\{\bar \alpha, 1-n/q\})$,} \label{hyp2}
\end{eqnarray}
which guarantee that the solutions of \eqref{defdiv} are continuously differentiable with $\alpha$-H\"older continuous gradient, according to the following classical and fundamental for the elliptic theory result (see \cite{Mo}, \cite{LU}, \cite{GT}). 

\smallskip

\noindent {\bf Theorem A.} {\it   Assume \eqref{hyp1}-\eqref{hyp2}, and ${\rm div}(b_1)+c\le 0$ in $\mathcal{D}'(\Omega)$. Then there is a unique solution $u\in H^1_0(\Omega)$ of  \eqref{defdiv},   $u\in C^{1,\alpha}(\Omega)$, and 
$$
\|u\|_{C^{1, \alpha}(\Omega)}\le C^\prime \|f\|_{L^q(\Omega)},
$$
for a constant $C^\prime$ 
which depends on $n, \lambda, \Lambda, q, \alpha, \bar \alpha$, diam$(\Omega)$, local $C^{1,\bar \alpha}$ charts describing $\partial\Omega$, and upper bounds for $\|A\|_{C^{\alpha}(\Omega)}$, $\|b_1\|_{C^{\alpha}(\Omega)}$, $\|b_2\|_{L^q(\Omega)}$, $\|c\|_{L^q(\Omega)}$.
}

\smallskip

The main goals of this work are:

\begin{itemize}

\item[(a)] to give explicit expressions of the constant $C^\prime$ in terms of $\Omega$ and the norms of the coefficients in \eqref{hyp2}; and, as a consequence, 
\vskip 1mm

\item[(b)] to prove upper and lower asymptotics for the first eigenvalue $\lambda_1$ of inversible elliptic operators with large coefficients or in large domains; and
    \vskip 1mm
    
\item[(c)] to obtain novel quantitative results on the Landis conjecture on exponential decay.
\end{itemize}

Motivated by the Landis conjecture, a specification of the constant in the classical gradient bound for the Dirichlet problem was recently given by Le Balc'h in \cite{LeB} (using also earlier ideas by C. Kenig and M. Pierre), for the particular equation with {\it bounded} coefficients
\begin{equation}\label{lebeq}
\Delta u + b(x)\cdot \nabla u + c(x)u = f(x),
\end{equation}
and $c\le 0$ in a smooth {\it convex} domain $\Omega$ with diam$(\Omega)\ge1$, whose solution $u\in H^1_0(\Omega)$ is thereby shown to satisfy
\begin{equation}\label{lebres}
 \|\nabla u\|_{L^\infty(\Omega)}
\le C_0 \exp\left\{C_0\left(1+\|b\|_{L^\infty(\Omega)} + \|c\|_{L^\infty(\Omega)}^{1/2}\right)\,\mathrm{diam}(\Omega)\right\}\|f\|_{L^\infty(\Omega)},
\end{equation}
for some $C_0=C_0(n)$. Note \eqref{lebeq} is  \eqref{defdiv} with $A=Id$, $b_1=0$, $q=\infty$. 
The proof in ~\cite{LeB} uses somewhat involved tools such as doubling of variables and existence of a positive ``multiplier" with logarithmic gradient bounds, which lead to the restrictions  on the operator and $\Omega$.\smallskip

In this paper we develop a novel and completely different approach, based on an optimized version 
of the global Harnack inequality from \cite{GSS} and a duality argument, which permit us to prove sharp estimates for the constant $C^\prime$ from Theorem A, valid for any operator in the most general form \eqref{defdiv} and arbitrary $C^{1,\alpha}$ domains, only under the assumptions on the coefficients \eqref{hyp1}-\eqref{hyp2} which guarantee that the solutions are continuously differentiable. In particular, unbounded lower-order coefficients are allowed. Even when restricted to \eqref{lebeq} with bounded coefficients in a ball, the results below are new and \eqref{lebres} will be improved,  uncovering some interesting phenomena in relation to the zero-order term. We note that our estimates also take explicitly into account the $C^{1,\alpha}$ norm of $\partial\Omega$ and the geometry of~$\Omega$.
Finally, we  give counterexamples showing that our gradient bounds are optimal, either asymptotically for large domains, or for large lower-order coefficients. 

Our treatment of unbounded coefficients is based on the use of locally uniform Lebesgue norms with radius adapted to the coefficients themselves, which  makes the resulting estimates  both optimal as well as invariant with respect to rescalings. \smallskip

Before discussing the specific implications of our bounds on  (b) and (c) above, we make a few brief comments on the general importance of $C^1$-estimates like \eqref{sharpC0alphapositeigintro}-\eqref{sharpC1alphapositeigintro2} below. Their theoretical interest is obvious;  as far as  applications are concerned, already in the particular case $f=1$ the solution $U\in H^1_0(\Omega)$ of $-\ld U =1$ (classically with $\ld=\Delta$) represents in mechanics or elasticity theory the  {\it torsion function} while its gradient measures the {\it shear stress}, so providing $C^1$ bounds for $U$ is essential. Similarly, in probability theory $U$ is the largest mean first exit time from $\Omega$ of the associated drift-diffusion process; in addition, the fundamental Donsker-Varadhan formula for the first eigenvalue (\cite{DV76}, \cite{LS17}) states that $\lambda_1(-\ld,\Omega)$ can be estimated from below by $\|U\|_{L^\infty(\Omega)}^{-1}$. The possibility to consider more general operators means allowing more realistic models with inhomogeneous media/diffusion and/or presence of drift/transport, and then of course bounds in terms of the coefficients which describe these phenomena are meaningful.

We next comment briefly on our results on the  problems described in (b) and (c) above (details and references will follow in the next section). 
Problem (b)  is a classical topic, 
especially in the framework of {\it operators with large drifts} such as \eqref{lebeq} with the drift $b(x)$ becoming large, going back to the works of Wentzell, Friedman and others in the 1970's,
and more recently studied in connection with population dynamics, propagation phenomena in biological problems
and shape optimization.
A key feature is that the absolute lower bound on $\lambda_1$ decays {\it exponentially} 
as a function of the size of the drift,  while the upper bound is {\it quadratic} in that size. 
Lower bounds are notoriously difficult to prove, and in previous works  have been obtained only for particular cases of \eqref{defdiv}, 
and under rather strong regularity or boundedness assumptions on the coefficients,
namely  $\ldu= \mathrm{div}(A\nabla u)+b \cdot \nabla u$ with at least 
$A\in W^{1,\infty}(\Omega)$ and $b\in L^\infty(\Omega)$.  Our first contribution  is to prove such upper and lower bounds for  \eqref{defdiv}, only under the assumptions \eqref{hyp1}-\eqref{hyp2}. Furthermore, our results yield  quantitative estimates on the monotonicity of the first eigenvalue with respect to the domain, specifically  when a neighbourhood of the boundary is added or removed.

As for item (c), an old question attributed to Landis asks whether
a nontrivial solution of a homogeneous,
uniformly elliptic PDE with bounded coefficients ($\Delta u + c(x)u=0$) in
the whole space 
cannot decay faster than exponentially at infinity, that is, whether
$\limsup_{|x|\to\infty} e^{K|x|}|u(x)|>0$
for some $K>0$ depending on bounds on the coefficients.
This problem has been studied extensively in the past 40 years, yet became  close to solved only very recently. The answer is positive (at least up to a logarithmic correction in the exponential) for  $n=2$;  
however, for $n\ge3$ there are solutions which decay like $\exp(-K|x|^{4/3})$. On the other hand, many positive answers were obtained under various supplementary assumptions on the coefficients which have in common to require at least that  
$\lambda_1(-\ld,\rn)\ge0$. Then a stronger result holds: $\sup_{\partial B_R}|u|\ge c_0e^{-KR}$, $R>R_0$. 
Also,  connections between this problem 
and optimized $C^1$ estimates were uncovered, but obtained
only for  particular cases of \eqref{defdiv}, and under  additional assumptions on the coefficients or the dimension. Our contributions to the Landis problem will be, on the one hand, to show that, under the usual and unavoidable to date assumption $\lambda_1(-\ld,\rn)\ge0$, the supremum in the  known estimate $\sup_{\partial B_R}|u|\ge c_0e^{-KR}$ can be replaced by the {\it average of $u$ on $\partial B_R$}; and on the other hand, that 
 {\it the assumption $\lambda_1(\ld,\rn)\ge0$ is actually  necessary in general for the latter estimate}.

\medskip

Thus, in a nutshell, we give both a novel approach leading to sharp gradient estimates in close to optimal generality, and a unified framework to the above three topics.

\section{Main results in a simplified setting}  
\label{SecResIntro}

In this section we state some of our main estimates in the case when $\Omega=B_R$ is a ball, and with somewhat relaxed constants, fully optimal only for bounded coefficients. We make this choice of presentation for the readers' convenience, since the most general results require some technical work which  would interfere with the readability and shortness of this introduction. Nevertheless, the following  theorems should be  sufficient for most applications, as well as  to grasp the essence of the more general statements in Section~\ref{SecGenRes} below.

\subsection{Gradient estimates} 
\label{SecResIntro1}

Our aim here is to specify the behavior of the constant $C^\prime$ in the classical a priori bound in Theorem A, when either the radius $R$ or the norms of the lower-order coefficients of $\ld$  become large.

 In what follows, $d=d(x)=\mathrm{dist}(x,\partial\Omega)$,  $C$ (possibly with indices) denote constants which depend only on $n,\lambda, \Lambda, q$, $\alpha, \bar\alpha$. Let $b=b_1+b_2$.
For $\alpha\in (0,1)$ and $n<q\le\infty$, we set 
 \begin{equation}\label{partoper_defM} 
M_0 = [A]_{\alpha, B_R}^{1/\alpha}+ \sigma_{q}\left(\|b\|_{L^q_{ul}(B_R)}\right),\qquad \mbox{with }\sigma_q(x):= \max\{x,x^{\frac{1}{1-n/q}}\},\;\sigma_\infty(x)=x, 
 \end{equation}
 \begin{equation}\label{partoper_defM1} 
M_1 = [A]_{\alpha, B_R}^{1/\alpha}+[b_1]_{\alpha,B_R}^{1/(\alpha+1)}
+ \|b_1\|_{L^\infty(B_R)} + \sigma_{q}\left(\|b_2\|_{L^q_{ul}(B_R)}\right) +\sigma_{2q}\left(\|c\|_{L^q_{ul}(B_R)}^{1/2}\right),
  \end{equation}
where $[A]_{\alpha,B_R}$ is the standard H\"older bracket in $B_R$ and for $h\in L^q(B_R)$ we denote 
\begin{equation}\label{defunifloc} 
\|h\|_{L^q_{ul}(B_R)}=\sup_{x_0\in B_R}\|h\|_{L^q(B_R\cap B_1(x_0))}, \qquad \|h\|_{L^\infty_{ul}(B_R)}=\|h\|_{L^\infty(B_R)}.
\end{equation}
These are the {\it uniformly local} Lebesgue norms, which have the advantage over the usual
 $L^q$-norms to measure $h$ only locally and to not deteriorate as the domain becomes large (for instance $\|1\|_{L^q(B_R)} = c(n,q)R^{n/q}$ while $\|1\|_{L^q_{ul}(B_R)} = c(n,q)$). They have been used in the study of  Navier-Stokes or parabolic equations, e.g.~in the classical papers \cite{Ka,GV} and more recently in \cite{HOS,IsSa,MaTe}, but are not usual in our context. Apart from improving the more frequently encountered bounds, the uniformly local norms will play an important role in the proofs below.

We have the following $C^{1,\alpha}$-estimates with explicit constants. 
\begin{thm} \label{thm1intro}
Let $R>0$.
Assume \eqref{hyp1}-\eqref{hyp2} with $\Omega=B_R$, and 
${\rm div}(b_1)+c\le 0$ in $\mathcal{D}'(B_R)$.
Then the unique solution $u\in H^1_0(B_R)$ of problem
\eqref{defdiv} is in $C^{1,\alpha}(B_R)$ and
\begin{equation}\label{sharpC0alphapositeigintro}
\|\nabla u\|_{L^\infty(\partial B_R)}\le \left\|\frac{u}{d}\right\|_{L^\infty(B_R)}\le C_0\,e^{C_0M_0R}\, R^{1-\frac{n}{q}}\|f\|_{L^q(B_R)},
\end{equation}
\begin{equation}\label{sharpC1alphapositeigintro}
\|\nabla u\|_{L^\infty(B_R)}
\le C_0\,(1+M_1R)\,e^{C_0M_0R}\, R^{1-\frac{n}{q}}\|f\|_{L^q(B_R)},
\end{equation}
\begin{equation}\label{sharpC1alphapositeigintro2}
[\nabla u]_{\alpha,B_R}
\le C_0\,(1+M_1R)^{1+\alpha}\,e^{C_0M_0R}\, R^{1-\frac{n}{q}-\alpha}\|f\|_{L^q(B_R)}.
\end{equation}
\end{thm}

\begin{rem} The somewhat unusual form of $M_0$ and $M_1$ (the use of $\sigma_q$) is entirely due to our desire to give quick-to-enunciate, yet sufficiently general, statements in this introduction. In Section~\ref{SecGenRes} we will give finer versions of $M_0$, $M_1$  which will render \eqref{sharpC0alphapositeigintro}-\eqref{sharpC1alphapositeigintro2} sharp also for finite $q>n$, and  scale-invariant in the sense that $\tilde M_i = RM_i$, if $\tilde M_i$ is the corresponding quantity for the operator in $B_1$ obtained from $\ld$ after the rescaling $x\to x/R$.

As we already noted, in the important case $b_1,b_2,c\in L^\infty(\Omega)$, i.e.~$q=\infty$ in \eqref{partoper_defM}-\eqref{partoper_defM1}, the bounds \eqref{sharpC0alphapositeigintro}-\eqref{sharpC1alphapositeigintro2}  are already scale-invariant as well as fully optimal (see below). 
\end{rem}

\begin{rem}
We observe that no explicit dependence on the zero order coefficient $c$ appears in \eqref{sharpC0alphapositeigintro}; 
this is due to the ``coercivity assumption" $c\le -{\rm div}(b_1)$.  
Extensions of Theorem~\ref{thm1intro} under more general coercivity assumptions will be given in Section~\ref{SecGenRes}.
\end{rem}

\begin{rem}
The constants in the global gradient bounds \eqref{sharpC1alphapositeigintro}-\eqref{sharpC1alphapositeigintro2} 
in Theorem \ref{thm1intro} are exponential with respect to $M_0$ but only polynomial with respect to the larger quantity $M_1$. As we already noted, the estimate \eqref{lebres}, 
i.e.~$\|\nabla u\|_\infty 
\le  e^{C_0{(1+M_1)R}}\,\|f\|_\infty$
 with $\|\cdot\|_\infty=\|\cdot\|_{L^\infty(B_R)}$,
was obtained by different methods in \cite{LeB} for the particular case \eqref{lebeq} with $b, c, f\in L^\infty(B_R)$, $c\le 0$; in addition, in \cite{LeB} it was suggested that \eqref{lebres} should be optimal. However, from \eqref{sharpC1alphapositeigintro} we see that the dependence in $\|b\|_\infty$ is indeed exponential, but it is only linear in $\|c\|_\infty^{1/2}$ (observe that in the case of \eqref{lebeq} and $q=\infty$ we have $M_0 =\|b\|_\infty$, $M_1 = \|b\|_\infty+ \|c\|^{1/2}_\infty$). 
\end{rem}

Broadly, the proof of Theorem \ref{thm1intro} is based on the following ideas: we first prove the Lipschitz bound in \eqref{sharpC0alphapositeigintro} by using a duality argument along with the global Harnack inequality from \cite{GSS} (in its optimized form given in \cite[Theorem 2.2]{SirSou2026a}). Then we combine this with an optimized $C^0$-to-$C^{1,\alpha}$ estimate, which is proved by localization and scaling. We state the latter in the following theorem (for its optimality see Remark \ref{gradopt}).

\begin{thm} \label{thm2intro1}
Let $R>0$ and assume \eqref{hyp1}-\eqref{hyp2} with $\Omega=B_R$.  If $u\in H^1_0(B_R)$ solves~\eqref{defdiv}, then $u\in C^{1,\alpha}(B_R)$ and
\begin{equation}\label{sharpC1genintro}
\|\nabla u\|_{L^\infty(B_R)}\le
C_0\Bigl({ \bigl(M_1+R^{-1}\bigr)} \|u\|_{L^\infty(B_R)}
+  R^{1-\frac{n}{q}} 
\|f\|_{L^q(B_R)}\Bigr),
\end{equation}
\begin{equation}\label{sharpC1alphagenintro}
[\nabla u]_{\alpha, \Omega}
\le  C_0\Bigl({\bigl(M_1+R^{-1}\bigr)}^{1+\alpha} \|u\|_{L^\infty(B_R)}
+  R^{1-\frac{n}{q}-\alpha}  
\|f\|_{L^q(B_R)}\Bigr).
\end{equation}
\end{thm}

We now turn to the sharpness of the estimates in Theorem~\ref{thm1intro}.
It can be exhibited even in the simplest case, when \eqref{defdiv} is $\Delta u + b(x)\cdot \nabla u + c(x) u = f(x)$ in $B_R$, with $b,c,f\in L^\infty(B_R)$.

The first  proposition deals with the  optimality of the dependence in the size of the domain; it states that we can find  a fixed equation defined on the whole of $\rn$  (we can even take $c=0$) with $b$ and~$f$ having uniformly bounded from above and below norms, which realizes the exponential dependence in the radius $R$  
 of the constant in  \eqref{sharpC0alphapositeigintro}-\eqref{sharpC1alphapositeigintro2} as $R\to\infty$.

\begin{prop} \label{Prop-Optim-Linfty-intro}
There exist $b\in L^\infty\cap C^\infty(\R^n)$ and a nontrivial $f\in L^\infty\cap C^\infty(\R^n)$ such that,
for each $R\ge 1$, $ n<q\le\infty$,
\be{optimLinfty2}
\min \left\{ \|\nabla u\|_{L^\infty(\partial B_R)}, \frac{u}{d}(0)\right\}\ge  e^{\hat c_0R} \ge c_0e^{c_0M_0R} R^{1-n/q}\|f\|_{L^q(B_R)},
\ee
 where $u>0$ is the classical solution of  $\Delta u+b\cdot\nabla u =f$ in $B_R$, $u=0$ on $\partial B_R$. In addition, $c_1\le M_0=\|b\|_{L^\infty(B_R)}\le C_1$ and $ \|f\|_{L^q(B_R)}=C_2$; here $c_0,\hat c_0, c_1,C_1,C_2>0$ depend only on $n$ and~$q$.
\end{prop}

We now turn to the dependence in the size of the coefficients of $\ld$. We start with the estimate \eqref{sharpC0alphapositeigintro}.
The next assertion exhibits, for each fixed $R>0$ (it will be enough to consider $R=1$, by scaling) and given   $K$, 
an operator as in \eqref{lebeq} whose drift coefficient $b$ has size $K$ and the solution of the Dirichlet problem at the origin together with its gradient on the boundary are of size which is exponential in $K$, that is, the constant in  \eqref{sharpC0alphapositeigintro} is sharp. In addition, $f$ can be taken to be uniformly bounded and nonpositive, and the zero-order coefficient can have arbitrarily large norm (note that when $c(x)\le0$ decreases, the positive solutions for a fixed $f$ decrease, by the maximum principle). 

\begin{prop}\label{Prop-Optim-Linfty-intro-largecn1}
 For any $R>0$, $K\ge0$, $L\ge0$, there exist $b,c,f\in L^\infty\cap C^\infty(B_R)$ such that
\be{hypCoeffOptim}
\hbox{$c,f\le0$ in $B_R$, \ \ $\|b\|_{L^\infty(B_R)}=K$, \ \ $\|c\|_{L^\infty(B_R)}^{1/2}=L$,\ \ 
$c_1\le R^{1-n/q}\|f\|_{L^q(B_R)}\le C_1$,}
\ee
 and the strong solution of  
 \be{EqOptim}
 \hbox{$\Delta u+b\cdot\nabla u +cu=f$, $u>0$  in $B_R$, $u=0$ on $\partial B_R$,}
 \ee
 satisfies
$$ 
\min \left\{ \|\nabla u\|_{L^\infty(\partial B_R)}, \frac{u}{d}(0)\right\}\ge c_0e^{c_0KR}\quad (\sim c_0e^{c_0M_0R}R^{1-n/q}\|f\|_{L^q(B_R)}).
$$
Here $c_0,c_1,C_1>0$  depend only on $n$ and $q$.
\end{prop}

We next discuss the optimality of the global estimate \eqref{sharpC1alphapositeigintro}, which
 is more delicate and, interestingly, depends on the dimension $n$.  We can prove that \eqref{sharpC1alphapositeigintro} is completely exact for $n\ge3$ and is exact up to a logarithmic correction in $\|c\|_\infty$ for $n=2$.
 On the contrary, it  turns out that \eqref{sharpC1alphapositeigintro} is not exact for $n=1$, and in that case the size of the gradient is bounded everywhere independently of the zero-order coefficient and obeys a similar 
bound as in~\eqref{sharpC0alphapositeigintro} -- see Proposition~\ref{Prop-Optim-Linfty-intro-largecndim1}.
Here is the precise statement for $n\ge2$.

\begin{prop} \label{Prop-Optim-Linfty-intro-largecn2}
 For each $R>0$, $K\ge0$, $L>0$, there exist $b,c,f\in L^\infty\cap C^\infty(B_R)$ such that
 \eqref{hypCoeffOptim} holds and
 the strong solution of  \eqref{EqOptim} satisfies
\begin{equation}
\label{gradn3} \|\nabla u\|_{L^\infty(B_R)} \ge c_0LR\,e^{c_0KR}\quad (\sim c_0(1+M_1R)\,e^{c_0M_0R}R^{1-n/q}\|f\|_{L^q(B_R)})\qquad\mbox{if } n\ge3,
\end{equation}
\begin{equation}
\label{gradn2} \|\nabla u\|_{L^\infty(B_R)}  \ge c_0\hat LR\,e^{c_0KR}\quad(\sim c_0(1+\hat LR)\,e^{c_0M_0R}R^{1-n/q}\|f\|_{L^q(B_R)})\qquad \mbox{if } n=2,
\end{equation}
where $\hat L= L\min\{1, |\log(L)|^{-1}\}$. 
 Here again $c_0,c_1,C_1>0$  depend only on $n$ and $q$.
A similar statement is valid for \eqref{sharpC1alphapositeigintro2}.
\end{prop}

As mentioned above, more comments on and  extensions of 
the results of this section will be given in Section~\ref{SecGenRes} 
(see Theorems~\ref{thm1gen}-\ref{thm3gen}) 
where we will treat the case of a general $C^{1, \alpha}$-domain~$\Omega$ and unbounded coefficients,
as well as more general coercivity assumptions than ${\rm div}(b_1)+c\le 0$.
 In this process, $R$ will be replaced by the geodesic diameter $D$ of $\Omega$
 and $M_0,M_1$ will be replaced by  refined quantities measuring the size of the coefficients 
by means of suitably adjusted uniformly local $L^q$ norms.

\subsection{Sharp estimates for the first eigenvalue}
\label{SecResIntro2}

Obtaining bounds for the eigenvalues of an elliptic operator (in particular for the lowest eigenvalue, whose existence and general properties are briefly discussed in the appendix below) is a fundamental problem, both in spectral theory and its applications.
To fix ideas, consider for instance the operator $L_ku:=A(x)D^2u+kb(x)\cdot \nabla u$ on a bounded $C^2$ domain~$\Omega$
under Dirichlet boundary conditions, where the matrix $A(x)$ is uniformly elliptic and $k>0$ is a (large) parameter,
 and  $A, b\in C^1(\overline\Omega)$.
In view of applications in probability, it was proved in \cite{Fri73} (see also \cite{Wen72, Wen75}) that if $b$
 satisfies the boundary outflow condition $b\cdot\nu>0$ on $\partial\Omega$ plus some additional conditions, then 
\be{estimFri}
e^{-C_1k}\le \lambda_1(-L_k,\Omega)\le e^{-C_2k},\quad k\to\infty,
\ee
 where $C_1,C_2$ are positive constants depending on $\Omega, A, b$.
Under additional assumptions on the vector field $b$, in relation with properties of the dynamical system $y'=b(y)$, 
further information on the exponential decay in \eqref{estimFri} was obtained in \cite{Ka79,JKN09}.
On the other hand, when the condition $(b\cdot\nu)_{|\partial\Omega}>0$ fails, different asymptotic behaviors of $\lambda_1(-L_k,\Omega)$
may occur. It was shown in \cite{DEF74} that, for any $A\in C^1(\overline\Omega)$ and 
$b\in C^\alpha(\overline\Omega)$, the absolute upper bound for the asymptotics of $\lambda_1$ is
\be{estimDEF}
\lambda_1(-L_k,\Omega)\le Ck^2,\quad k\to\infty,
\ee
 where $C>0$ depends on $\Omega, A, b$,
and that $\lambda_1(-L_k,\Omega)\sim k^2$ under suitable conditions on~$b$.
Intermediate asymptotic behaviors between $e^{-Ck}$ and $Ck^2$ were later obtained
under various assumptions on $b$ (see, e.g.,~ \cite{DEF74,Ki80,Ce81,EK87,BHN05}).
For results on asymptotic properties of the corresponding eigenfunction, we refer to \cite{DF78,BHN05,HRR18,GLZ25,GLZ26} and the references therein.

More recently, this topic has been studied in connection with population dynamics and propagation phenomena in biological problems
(see, e.g.,~\cite{BHN05,CL08,CL12}) and with shape optimization.
Regarding the latter, the work \cite{HNR11} studies the minimization problem for $\lambda_1$
under various constraints, including $\|b\|_\infty=K$,
$|\Omega|=\mu$ and $\Lambda_{min}(A)=1$ (least eigenvalue of $A$, assumed to be symmetric),
for given $K,\mu>0$.
Assuming $A\in W^{1,\infty}(\Omega)$ and $b\in L^\infty(\Omega)$,
by suitable rearrangement methods, in \cite{HNR11} the authors show some Faber-Krahn type inequalities which in particular guarantee
that $\lambda_1$ is minimal when $\Omega$ is a ball of measure $\mu$, $b=-K\ts\frac{x}{|x|}$ and $A=Id$.
The first eigenvalue  $\tilde\lambda_1$ of the latter (radially symmetric) problem can be precisely estimated by ODE methods,
which leads to the sharp lower bound
\be{estimHNR}
\lambda_1\ge \tilde\lambda_1=\exp\big[-(1+o(1))\|b\|_\infty R\big],\quad \|b\|_\infty\to\infty,
\ee
showing in particular that the exponential behavior in \eqref{estimFri} essentially represents the absolute lower bound for the asymptotics of $\lambda_1(-L_k,\Omega)$ as $k\to\infty$. We note that the rearrangement methods require the boundedness of the drift $b$.

Indeed, as far as we know, all the existing results on lower bounds for the first eigenvalue 
for large drifts require 
rather strong smoothness assumptions on the principal coefficients, namely at least 
$A\in W^{1,\infty}(\Omega)$ -- which in particular entails that
the operator is of both divergence and nondivergence form -- as well as boundedness of the drift $b\in L^\infty(\Omega)$.
As an application of  Theorem~\ref{thm1intro} and its extensions,
we are  able to obtain an exponential lower bound on $\lambda_1$,
which is valid for the general class of 
genuine divergence form operators in \eqref{defdiv}, with possibly unbounded coefficients.

\begin{thm} \label{thm4b}
Let $R>0$. Assume \eqref{hyp1}-\eqref{hyp2} with $\Omega=B_R$, and
${\rm div}(b_1)+c\le 0$ in $\mathcal{D}'(B_R)$ (note drift operators have $b_1=c=0$). 
Let $M_0$ be defined in \eqref{partoper_defM}.
Then
\be{estimlambda1intro}
\lambda_1(-\mathcal{L},B_R)\ge c_0R^{-2} e^{-C_0M_0R}.
\ee
\end{thm}

For drift operators of the form $\ld=\Delta+b\cdot\nabla$, the optimality of estimate \eqref{estimlambda1intro} for fixed $R>0$
with respect to $M_0=\|b\|_{L^\infty}$
clearly follows from \eqref{estimFri}. 
We complement this observation with the following proposition, 
which exhibits an operator defined on the whole of $\rn$,  with bounded coefficient, 
that realizes the optimality of estimate \eqref{estimlambda1intro} as $R\to\infty$.

\begin{prop} \label{Optim-spectral1} 
There exists $b\in L^\infty\cap C^\infty(\R^n)$ such that the operator $\ld=\Delta+b\cdot\nabla$ satisfies, for each $R\ge1$,
\be{optim-spectralineq1}
0\le \lambda_1(-\mathcal{L},\rn) \le\lambda_1(-\mathcal{L},B_R) \le CR^{-2} e^{-C_0M_0 R}
\ee
and $C_1\le M_0\le C_2$, with $M_0:=\|b\|_{L^\infty(B_R)}$, where $C, C_0,C_1,C_2>0$ depend only on $n$.
\end{prop}

Furthermore, we can also extend the absolute upper bound \eqref{estimDEF} to our general setting. 

\begin{thm}\label{absupperlambda1}
Let $R>0$. Assume \eqref{hyp1}-\eqref{hyp2} with $\Omega=B_R$. Let $M_1$ be defined in~\eqref{partoper_defM1}. Then 
\begin{equation}\label{upperbdlam1}
\lambda_1(-\mathcal{L},B_R) \le C_0 (M_1 +R^{-1})^2.
\end{equation}
\end{thm}
For operators in non-divergence form with bounded coefficients, the estimate \eqref{upperbdlam1} was proved in \cite[Lemma 1.1]{BNV}. A more general version of Theorem  \ref{absupperlambda1} is proved in the appendix; it will  be needed in the proof of Theorem \ref{thm1intro} and its extensions. \smallskip

The a priori bounds in Section \ref{SecGenRes} below also imply 
a sharp estimate on the strict monotonicity of $\lambda_1$ with respect to the domain. We illustrate it here in a simple situation (more general results can readily be deduced from the theorems and arguments in Section \ref{SecGenRes}). 
\begin{thm} \label{thmmonot}
Let $R\ge1$ and $\kappa\in (0,R/2)$. Assume \eqref{hyp1}-\eqref{hyp2} with $\Omega=B_R$ and $q=\infty$. Let $M_1$ be defined in \eqref{partoper_defM1}.
Then
\be{estimlambda1intromonot}
\lambda_1(-\mathcal{L},B_{R-\kappa})\ge \lambda_1(-\mathcal{L},B_R) + c_0R^{-2} e^{-C_0(M_1+\kappa^{-1})R}.
\ee
\end{thm}
 For operators in non-divergence form with bounded coefficients, a more general in terms of the removed domain, but only qualitative, estimate of this type can be found in \cite[Theorem~2.4]{BNV}. Theorem \ref{thmmonot} quantifies that result when we remove a 
neighborhood 
of the boundary.  For the sharpness of \eqref{estimlambda1intromonot} see Remark \ref{remdifsharp}.

Extensions of the above eigenvalue bounds, for general domains $\Omega$, with refined uniformly local $L^q$ norms, as well as 
with weaker coercivity assumptions than  ${\rm div}(b_1)+c\le 0$,
 will be given in following sections
(see Theorems~\ref{thm1gen}(ii)-\ref{thm3gen}(iii) and Proposition \ref{upperboundlambda1}).

\subsection{Quantitative estimates related to  the Landis conjecture}
\label{SecResIntro3}

The following question is attributed to Landis and formulated in \cite{KL}: can a nontrivial solution of a uniformly elliptic PDE with bounded coefficients in
$\mathbb{R}^n$  decay faster than exponentially at infinity:
\be{limsupexp}
\limsup_{|x|\to\infty} e^{K|x|}|u(x)|>0,
\ee
for some $K>0$ depending on bounds on the coefficients~?
This property is known as ``the Landis conjecture", also in its sharper form where the optimal $K$ is sought for,
or in a weaker form proposed in \cite{K2},
where the decay to rule out is $\exp(-|x|^{1+\epsilon})$, $\epsilon>0$. The question was answered in the negative long ago for complex solutions in \cite{M}: there exist such solutions with $|u(x)|\le C\exp\left(-c|x|^{4/3}\right)$; however, the problem remained open for real solutions until very recently.

The Landis conjecture has a long history, we refer to \cite{KL, M, BK, EK, K2, DW, B1, 
KSW, DKW, KW, DZ, R, ABG, LM, LeB, LeBS, DP, FI} and the references in these works. It was  proved in \cite{LM} that if $n=2$ and a real solution satisfies $|u(x)|\le C\exp\left(-|x|(\log|x|)^{1/2}\right)$ then it is trivial. On the other hand, for $n\ge 3$ a  nontrivial real solution such that $|u(x)|\le C\exp\left(-c|x|^{4/3}\right)$ was constructed in \cite{FI}. This brings forward the question what would be the best supplementary hypothesis under which the Landis conjecture holds in any dimension. Many of the above quoted works give partial answers to that question, and have in common that they require at least that 
  $\mathcal{L}$ satisfies the maximum principle in all bounded subdomains of $\rn$ (i.e. $\lambda_1(\ld,\rn)\ge0$). In \cite{SS2},
we proved the Landis conjecture only under that hypothesis, and  with the stronger estimate \begin{equation}\label{concllandis2}
 \sup_{|x|=R} |u(x)|\ge e^{-C_0KR},\quad R\ge R_0.
\end{equation}
This result  is valid
for both divergence and non-divergence operators,
with unbounded lower-order coefficients which are only uniformly locally integrable (and thus bounded coefficients are a very particular case).
To prove \eqref{concllandis2} we introduced a new approach based on a Harnack inequality
with sharp dependence on the coefficients and the size of the domain.

Our contribution here is twofold. First, we show that the hypothesis on the validity of the maximum principle is actually necessary  in general  for the bound \eqref{concllandis2}; and second, 
we obtain a  more precise quantitative estimate than \eqref{concllandis2}, namely
\be{intexp}
\int_{\partial B_R}\, |u| d\sigma \ge e^{-C_0 KR},\quad R\ge R_0.
\ee
It was observed in \cite{LeB} that such an estimate can be obtained via a duality argument;
however, the results and proofs in \cite{LeB} are restricted to  $A=Id$, $b_2=0$, $b_1, c$ bounded with~$c\le 0$.

We will combine the Harnack inequality approach from \cite{SS2}, the duality method from \cite{LeB} and Theorem \ref{thm1intro}
to prove \eqref{intexp} for the general class of divergence operators \eqref{defdiv} with unbounded coefficients,
under
the usual assumption that the first eigenvalue of $\mathcal{L}$ is positive on all bounded subdomains (which we now know is also necessary for \eqref{intexp}; and is of course much more general than $ b_2=0, c\le0$).
Apart from applying to very general equations, the Harnack inequality approach considerably simplifies the duality method from~\cite{LeB}. In particular, we dissociate the proof of the Landis conjecture and \eqref{intexp} from the
use of the so-called ``multiplier'' and the logarithmic gradient estimates for this multiplier, which appear as  essential steps in many previous works on the Landis conjecture for real solutions 
(see for instance \cite{KSW, DW, LeB}).

For simplicity of notation, we again give a typical statement
under strengthened assumptions on the coefficients (a more general result will be given in Section~\ref{SecGenRes}, see Theorem~\ref{ThmLandisDualityB}).

\begin{thm} \label{ThmLandisDualityIntro}
Assume \eqref{hyp1} with $\Omega=\rn$, and
\be{hyp2intro}
A, b_2\in C^\alpha(\rn),\  b_1,c\in L^q_{ul}(\rn)\quad\hbox{for some $\alpha\in(0,1)$, $n<q\le \infty$.}
\ee
Let $M_1$ be defined in \eqref{partoper_defM1} with $B_R$ replaced by $\rn$. If $\lambda_1(-\mathcal{L},\R^n)\ge 0$, which is true for instance when ${\rm div}(b_1)+c\le 0$ in $\mathcal{D}'(\rn)$,
then
any  weak solution  of $\mathcal{L}u=0$ in $\R^n$ satisfies the lower estimate
$$\int_{\partial B_R}\, |u| d\sigma \ge c_0e^{-C_0M_1R} \int_{B_R} |u|,\quad R\ge 1.$$
\end{thm}

As noted above, a remarkable fact
 about \eqref{intexp} is that the positivity of the first eigenvalue turns out to be {\it necessary}
 in general for its validity (and a fortiori for the validity of \eqref{concllandis2}). 

\begin{prop} \label{PropOptimality}
For every $\lambda<0$, there exist bounded coefficients $b, c \in C^\infty(\R^n)$
such that
the operator $\mathcal{L}:=\Delta+b\cdot\nabla+c$ satisfies
$\lambda_1(-\mathcal{L},\R^n)=\lambda$
and $\mathcal{L}u=0$ admits a classical solution $u$ on $\R^n$ such that,
for some sequence $R_i\to\infty$ and some  $ K>0$,
$$u(x)\equiv 0\quad\hbox{on $|x|=R_i$}, \qquad
\mbox{and}\qquad \limsup_{|x|\to\infty} e^{K|x|}|u(x)|=1.$$
\end{prop}

 We quickly comment on the organization of the paper.  
In Section~\ref{sec-prelim},  after giving the necessary notation, we provide some auxiliary results and important ingredients of our proofs,
especially optimized Harnack inequalities,
as well as properties associated with uniformly local norms.
 Section~\ref{SecGenRes} is devoted to the extensions and proofs of the  main estimates. In Section \ref{SecGenRes4} we give various examples which establish the sharpness of  the gradient and eigenvalue estimates, and prove Proposition \ref{PropOptimality}. 
Finally, in the appendix we give some properties of the first eigenvalue, 
as well as a suitable integration by parts formula which we need.

\section{Notations and auxiliary results} \label{sec-prelim}

\subsection{Notations}

\label{SecNot}

We recall  we always assume that
\be{hypbaaralpha}
\bar\alpha\in (0,1], \quad n<q\le \infty, \quad 0<\alpha< \min\{1-n/q,\bar\alpha\},
\ee
  and we denote with $c,C>0$ (possibly with indices) generic positive constants which may depend only on $n,\lambda,\Lambda, q,\alpha,\bar\alpha$, and may change from line to line. 
 We also assume that $\Omega$ is a bounded domain with $C^{1,\bar\alpha}$ boundary, and set
\be{defD0D}
D_0:={\rm diam}(\Omega),\quad D:=\mbox{ geodesic diameter of }\Omega \;= \sup_{x,y\in \Omega} \inf_{\sigma\in\Sigma(x,y)} l(\sigma),
\ee
where $\Sigma(x,y)$ is the set of all paths in $\Omega$ connecting $x$ and $y$, and $l(\sigma)$ is the length of $\sigma$.
 The distance to the boundary of $\Omega$ is denoted by $d=d(x)=\mathrm{dist}(x,\partial\Omega)$.

Next we define a constant $r_\Omega$ which quantifies the well-known facts that each point of the boundary of a $C^{1,\bar\alpha}$-domain: (i) has a uniform neighborhood in which the domain can be ``flattened", and (ii) can be touched by a $C^{1,\bar\alpha}$-paraboloid with fixed size and opening.

Let $\rho_\Omega>0$ be the supremum of all
$r\le D_0$
such that, for each $\xi\in \partial\Omega$, there is a $C^{1,\bar\alpha}$-diffeomorphism $\Phi$ between the domain $(\Omega\cap B_{r}(\xi) - \xi)/r$ (resp. the surface $(\partial \Omega\cap B_{r}(\xi) - \xi)/r$) and $B_1^+=B_1(0)\cap\{x_n>0\}$ (resp. $B_1^0 = B_1(0)\cap\{x_n=0\}$) such that $D\Phi(0)=I$, $ |D\Phi| \le 2$, $ |D\Phi^{-1}| \le 2$, $[D\Phi]_\alpha\le1$. The existence of such $r>0$ is well known, see for instance \cite{Liebe}.
Also, there exist $\bar\rho_\Omega, k_\Omega>0$ with the following property:
for each $\xi\in \partial\Omega$,
there exists an orthonormal coordinate system $y=(y',y_n)\in\R^{n-1}\times\R$ with origin $\xi$ and a $C^{1,\bar\alpha}$ function
$\varphi$ on $B'_0=\{|y'|<\bar\rho_\Omega\}$, such that,
setting $\Sigma_0=B'_0\times(-\bar\rho_\Omega,\bar\rho_\Omega)$, we have
$\Sigma_0\cap\Omega=\Sigma_0\cap \{y_n>\varphi(y')\}$,
$\Sigma_0\cap\partial\Omega=\Sigma_0\cap \{y_n=\varphi(y')\}$,
$\varphi(0)=\xi$, $D\varphi(0)=0$ and $[D\varphi]_{\bar\alpha}\le k_\Omega$; hence
$$\bigl\{y=(y';y_n)\in \R^{n-1}\times\R:\ |y'|< \bar\rho_\Omega\ \hbox{and}\ k_\Omega |y'|^{1+\bar\alpha}<y_n< \bar\rho_\Omega\bigr\}\subset \Omega.$$
For a very general result on these geometrical considerations see \cite{AM}, in particular Corollary~3.14 in that paper.
We  use the following quantity from \cite{SirSou2026a} (related to the sharp Harnack inequality recalled in Theorem \ref{BHIoptim} below)
\be{def_romega}
r_\Omega=\min\bigl(\rho_\Omega, \bar\rho_\Omega/4,(120 k_\Omega)^{-1/\bar\alpha})\bigr).
\ee
 It is easy to see that if $\Omega=B_R$ (as in the introduction) then we can take $\bar\alpha=1$, and 
\begin{equation}\label{scaleinvtilde1b}
r_{B_R}=c(n)R,\quad D=D_0=2R,
\end{equation}

 We  denote by $\|\cdot\|_{L^q(\Omega)}$ the usual Lebesgue norm,
by $\|\cdot\|_{L^q_d(\Omega)}$ the Lebesgue norm weighted by $d$,
and by $[\cdot]_{\alpha,\Omega}$ the usual H\"older seminorm on~$\Omega$.
We will use the following  uniformly local Lebesgue norm: for  $q\in[1,\infty]$, $r>0$,
\begin{equation}\label{deful}
\|h\|_{q,r,\Omega}:=\sup_{x\in \Omega} \|h\|_{L^q(\Omega\cap B_{{r}}(x))},
\quad h\in L^q(\Omega),
\end{equation}
(of course $\|h\|_{\infty,r,\Omega}=\|h\|_{L^{\infty}(\Omega)}$ for any $r>0$), and the uniformly local H\"older bracket
\begin{equation}\label{defHbracket}
[\psi]_{\alpha, r,\Omega} := \sup_{x\in\overline{\Omega}}\ \sup_{y,z\in B_{r}(x)\cap \Omega} |y-z|^{-\alpha}|\psi(y)-\psi(z)|,
\quad \psi\in C(\overline\Omega),\quad  \alpha\in(0,1).
\end{equation}

Given an operator $\ld$ as in \eqref{defdiv} satisfying \eqref{hyp1}-\eqref{hyp2},
 we define the  quantity 
\begin{equation}\label{defM}
M=M(\ld,\Omega) =
[A]_{\alpha,r_0,\Omega}^{\frac{1}{\alpha}} +  [b_1]_{\alpha,r_0,\Omega}^{\frac{1}{1+\alpha}}+ 
\bigl\|b_1\bigr\|_{L^\infty(\Omega)} +\bigl\|b_2\bigr\|^{\beta_q}_{q,r_0,\Omega}+\|c\|^{\gamma_q}_{q,r_0,\Omega},
\end{equation}
$$\mathrm{where}\quad\beta_q=\frac{1}{1-{n}/{q}},\quad \gamma_q=\frac{1}{2-{n}/{q}} \qquad \left(\beta_\infty=1, \quad \gamma_\infty = {1}/{2}\right)$$
on which the estimates below depend, and the number $r_0 \,=r_0(\ld,\Omega) \in (0,r_\Omega]$ by 
\begin{equation}\label{defr0}
r_0:= \sup\Bigl\{r>0:\: r\bigl(r^{-1}_\Omega+ [A]^{\frac{1}{\alpha}}_{\alpha, r, \Omega} + [b_1]^{\frac{1}{1+\alpha}}_{\alpha, r, \Omega}+ \|b_1\|_{L^\infty(\Omega)} +\|b_2\|^{\beta_q}_{q,r,\Omega} + \|c\|^{\gamma_q}_{q, r, \Omega}\bigr)\le 1\Bigr\}.
\end{equation} 
where $r_\Omega$ is defined in \eqref{def_romega}. 
 We note that $r_0$ is the point where the increasing in $r$ function in the parentheses in \eqref{defr0} meets the decreasing function $1/r$. 
 See also Proposition \ref{basicr0}.

 For some results we will also use 
the following weaker variant of~$M, r_0$:
\be{defMstar}
M_*=M_*(\ld,\Omega)=[A]^{\frac{1}{\alpha}}_{\alpha, r_*, \Omega} +
\|b_1+b_2\|^{\beta_q}_{q,r_*,\Omega},\qquad\mbox{where}
\ee
\be{defrstar}
r_*=r_*(\ld,\Omega)= \sup\bigl\{r>0:\: r\bigl(r^{-1}_\Omega+ [A]^{\frac{1}{\alpha}}_{\alpha, r, \Omega} +
\|b_1+b_2\|^{\beta_q}_{q,r,\Omega}\bigr)\le 1\bigr\}.
\ee

For readers' convenience, we immediately state convenient upper bounds for $M$ and $M_*$ in which the implicit radii $r_0, r_*$ are replaced by $1$ (note that $\|\cdot\|_{q,1,\Omega}=\|\cdot\|_{L^q_{ul}(\Omega)}$ as defined in the beginning of Section \ref{SecResIntro}). These relaxed bounds should  be sufficient for most applications; 
however the full optimality of the gradient estimates in the case of unbounded coefficients is guaranteed only with $r_0$, $r_*$.

\begin{prop}\label{convenientbounds}
 We have  the upper estimates
\begin{equation}\label{defMmajor}
M(\ld,\Omega) \le
1+[A]^{\frac{1}{\alpha}}_{\alpha,\Omega}+ [b_1]^{\frac{1}{1+\alpha}}_{\alpha,\Omega} + \bigl\|b_1\bigr\|_{L^\infty(\Omega)} + \bigl\|b_2\bigr\|^{\beta_q}_{q,1,\Omega}+\|c\|^{\gamma_q}_{q,1,\Omega},
\end{equation}
\begin{equation}\label{defMmajor2}
M_*(\ld,\Omega) \le
1+[A]^{\frac{1}{\alpha}}_{\alpha,\Omega} + \|b_1+b_2\|^{\beta_q}_{q,1,\Omega}.
\end{equation}
that is, 
$r_0, r_*$ can be replaced by $1$, provided  $M$, $M_*$ are replaced by $M+1$, $M_*+1$
and the uniformly local H\"older brackets are replaced by the usual  ones.

Also, for $\sigma_q$ defined in \eqref{partoper_defM} and some $C=C(n,q)>0$ we have 
\begin{equation}\label{defMmajorsigma1}
M(\ld,\Omega) \le C\left(
[A]^{\frac{1}{\alpha}}_{\alpha,\Omega}+ [b_1]^{\frac{1}{1+\alpha}}_{\alpha,\Omega} + \bigl\|b_1\bigr\|_{L^\infty(\Omega)} + \sigma_q(\bigl\|b_2\bigr\|_{q,1,\Omega})
+\sigma_{2q}(\|c\|_{q,1,\Omega}^{1/2})\right),
\end{equation}
\begin{equation}\label{defMmajorsigma2}
M_*(\ld,\Omega) \le  C\left(
[A]^{\frac{1}{\alpha}}_{\alpha,\Omega}+ \sigma_q(\|b_1+b_2\|_{q,1,\Omega})\right).
\end{equation}
\end{prop}

Next, we give some more properties of the quantities defined in \eqref{defM}-\eqref{defrstar}, and prove Proposition \ref{convenientbounds}.

\subsection{Properties associated with uniformly local norms}

In this subsection we recall some useful properties,  stated in \cite[Section 3]{SirSou2026a} and proved in  the appendix of that paper, of
the  quantities $r_0, M$ and the associated uniformly local norms.

\begin{prop} [\cite{SirSou2026a}]
\label{basicr0}
 Recalling definitions \eqref{def_romega}, \eqref{defM}-\eqref{defrstar}, we have
\begin{equation}\label{relMr0}
r_0^{-1}=r^{-1}_\Omega+ M,
\end{equation}
\be{relMstarr0}
r_*^{-1}=r^{-1}_\Omega+M_*,
\ee
In particular
$\textstyle\frac12 \min\bigl(r_\Omega,M^{-1}\bigr)\le r_0\le\min\bigl(r_\Omega,M^{-1}\bigr)$,
and similarly for $r_*$.
\end{prop}

Next, we give a simple monotonicity property for general domains.

\begin{prop}[\cite{SirSou2026a}]\label{scaleinvtilde0}
Let $\theta\in(0,1]$, $\omega\subset\Omega$ be bounded domains with $C^{1,\bar\alpha}$ boundaries, and 
$\mathcal{L}$ be an operator as in \eqref{defdiv} satisfying \eqref{hyp1}, \eqref{hyp2}. There exists $C=C(n,q,\alpha,\bar\alpha,\theta)>0$ such that,
 if $r_\Omega\ge\theta r_\omega$, then
$M(\ld,\omega)\le CM(\ld,\Omega)$.
\end{prop}

We now consider the case of balls.
It is immediate that, if $u$ is a solution of $\mathcal{L}u=f$ in $\Omega=B_R$, then $\tilde u(x)=u(Rx)$ is a solution of
$\tilde{\mathcal{L}}u=\tilde f$ in $B_1$,
where $\tilde f$ and the coefficients of $\tilde{\mathcal{L}}$ are given by
\begin{equation}\label{scaledcoeff}
\tilde A(x)=A(Rx),\quad \tilde b_i(x)=Rb_i(Rx),\quad \tilde c(x)=R^2c(Rx),\quad \tilde f(x)=R^2f(Rx).
\end{equation}
The next proposition, after giving a basic monotonicity property of the quantity $M$ with respect to $R$, shows that
$r_0, M$, as well as each of the (uniformly local)
norms and seminorms appearing in $M$
enjoy some natural invariance properties with respect to the scaling transformation \eqref{scaledcoeff}.
As a consequence, all main estimates in the following Section~\ref{SecGenRes}
are invariant with respect to rescalings $x\to x/R$, and
in particular the general case $R>0$ can be reduced to the case $R=1$.

\begin{prop}[\cite{SirSou2026a}]\label{scaleinvtilde} 
Let $R>0$ and assume \eqref{hyp1}-\eqref{hyp2} with $\Omega=B_R$.
\smallskip

{(i) We have 
\be{MrhoMR}
M(\mathcal{L},B_\rho)\le M(\mathcal{L},B_R),\quad \rho\in(0,R).
\ee

(ii) Let $f\in L^q(B_R)$} and let $\tilde f$ and the coefficients of the operator $\tilde{\mathcal{L}}$ be given by \eqref{scaledcoeff}.
Let $\tilde r_0$ be defined by \eqref{defr0} with $\mathcal{L},B_R$
replaced by $\tilde{\mathcal{L}},B_1$.
Then
\begin{equation}\label{scaleinvtilde2}
\tilde r_0=r_0/R,
\end{equation}
\begin{equation}\label{scaleinvtilde2b}
[\tilde A]_{\alpha, \tilde r_0, B_1}= R^{\alpha}[A]_{\alpha, r_0, B_R},\,
[\tilde b_1]_{\alpha, \tilde r_0, B_1}= R^{1+\alpha}[b_1]_{\alpha, r_0, B_R},\,
 \|\tilde b_1\|_{L^\infty(B_1)}=R\|b_1\|_{L^\infty(B_R)},
\end{equation}
\begin{equation}\label{scaleinvtilde4}
\|\tilde b_2\|_{q,\tilde r_0,B_1}=R^{1-\frac{n}{q}}\|b_2\|_{q,r_0,B_R},\
\|\tilde c\|_{q,\tilde r_0,B_1}=R^{2-\frac{n}{q}}\|c\|_{q,r_0,B_R},\
\|\tilde f\|_{q,\tilde r_0,B_1}=R^{2-\frac{n}{q}}\|f\|_{q,r_0,B_R}.
\end{equation}
Consequently,
\begin{equation}\label{scaleinvtilde5}
M(\tilde{\mathcal{L}},B_1)=RM(\mathcal{L},B_R).
\end{equation}
\end{prop}

\begin{proof}[Proof of Proposition \ref{convenientbounds}]
(i) Inequalities \eqref{defMmajor}-\eqref{defMmajor2} follow from 
\eqref{relMr0} and \eqref{relMstarr0}. 
Indeed, if $r_0\le 1$, then  $\|\cdot\|_{q,r_0,\Omega}\le \|\cdot\|_{q,1,\Omega}$
and $[\cdot]_{\alpha, r_0, \Omega}\le [\cdot]_{\alpha, 1, \Omega}$,
whereas if $r_0\ge 1$, then $M\le 1$ owing to \eqref{relMr0},
hence \eqref{defMmajor}, and the argument is similar for \eqref{defMmajor2} and \eqref{relMstarr0}.

\smallskip

(ii) We write
$M=M''+M'$, with
$$M'= \bigl\|b_2\bigr\|^{\beta_q}_{q,r_0,\Omega}+\|c\|^{\gamma_q}_{q,r_0,\Omega}$$
and
\be{boundM2prime}
M''=[A]_{\alpha,r_0,\Omega}^{\frac{1}{\alpha}} +  [b_1]_{\alpha,r_0,\Omega}^{\frac{1}{1+\alpha}}+ 
\bigl\|b_1\bigr\|_{L^\infty(\Omega)}\le [A]_{\alpha,\Omega}^{\frac{1}{\alpha}} +  [b_1]_{\alpha,\Omega}^{\frac{1}{1+\alpha}}+ 
\bigl\|b_1\bigr\|_{L^\infty(\Omega)}.
\ee
	Since $r_0\le 1$ implies $M' \le \|b_2\|_{q,1,\Omega}^{\beta_q}
	+ \|c\|_{q,1,\Omega}^{\gamma_q}$,
we may assume $r_0\ge 1$.
We next split the discussion into the following cases:

\vskip 1pt
$\bullet$ {\it Case 1:} $M''\ge M'$. Then $M \le 2 M''.$

\vskip 1pt
$\bullet$ {\it Case 2:} $M''\le M'$. Then $M \le 2 M'.$
We may cover any ball of radius $r_0$ by $Cr_0^n$ unit balls and we further split into 
two subcases.

\vskip 2pt
\ \quad $\ast$ {\it Case 2.1:} $\|b_2\|_{q,r_0,\Omega}^{\beta_q} \ge \|c\|_{q,r_0,\Omega}^{\gamma_q}$.
By \eqref{relMr0} we get
$$M\le 2M' \le 4 \|b_2\|_{q,r_0,\Omega}^{\beta_q}
\le C \big[r_0^{n/q} \|b_2\|_{q,1,\Omega} \big]^{1/(1-n/q)} 
\le C  \big[M^{-n/q} \|b_2\|_{q,1,\Omega} \big]^{1/(1-n/q)},$$
hence $M\le C \|b_2\|_{q,1,\Omega}$.

\vskip 2pt
\ \quad $\ast$ {\it Case 2.2:} $\|b_2\|_{q,r_0,\Omega}^{\beta_q} \le \|c\|_{q,r_0,\Omega}^{\gamma_q}$. Then
$$M\le 2M' \le 4 \|c\|^{\gamma_q}_{q,r_0,\Omega} \le C  \big[r_0^{n/q}\|c\|_{q,1,\Omega} \big]^{1/(2-n/q)} \le C  \big[M^{-n/q} \|c\|_{q,1,\Omega} \big]^{1/(2-n/q)},$$
hence $M\le C \|c\|_{q,1,\Omega}^{1/2}.$

In all cases we obtain
	$$M \le 2 M'' +C  \max\Big\{\|b_2\|_{q,1,\Omega}^{\beta_q},  \|c\|_{q,1,\Omega}^{\gamma_q},  
	\|b_2\|_{q,1,\Omega},  \|c\|_{q,1,\Omega}^{1/2}\Big\},
	$$
and \eqref{boundM2prime} and the definition of $\sigma_q$ yield \eqref{defMmajorsigma1}.
The proof of \eqref{defMmajorsigma2} is similar.
\end{proof}

\subsection{Optimized Harnack inequalities}
\label{precisedef}

We will make important use of the following global Harnack inequality with optimized constants from our recent paper \cite{SirSou2026a}, 
which in turn is an optimized version of a result from \cite{GSS}.

\begin{thm}[\cite{SirSou2026a}] \label{BHIoptim}
 Let $\Omega$ be a bounded $C^{1,\bar\alpha}$-domain of $\mathbb{R}^n$ with  geodesic diameter $D$, and let  $r_\Omega$ be defined in \eqref{def_romega}. Assume \eqref{hyp1}-\eqref{hyp2}.
There exist  constants $\epsilon, C_0>0$ depending only on $n,q,\alpha,\lambda,\Lambda$, such that if $u\ge0$ in $\Omega$ is a weak solution of $-\mathcal{L}u= f $ in $\Omega$, $u=0$ on $\partial\Omega$, for some $f\in L^q(\Omega)$,
then
\begin{equation}\label{sharpBHI}
\sup_{\Omega} \frac{u}{d}\le
e^{C_0(r^{-1}_\Omega+ M)D}\left( \inf_{\Omega} \frac{u}{d} +
D^{1-\frac{n}{q}} \|f\|_{L^q(\Omega)}\right).
\end{equation}
where $M=M(\ld,\Omega)$ is defined in \eqref{defM}. 
\end{thm}

 In the particular case $\Omega=B_R$,  by \eqref{scaleinvtilde1b} the estimate  \eqref{sharpBHI}  becomes
\begin{equation}\label{sharpBHI2}
\sup_{B_R}\frac{u}{d}\le
e^{C_0(1+MR)} \left( \inf_{B_R} \frac{u}{d} +
R^{1-\frac{n}{q}}  \|f\|_{L^q(B_R)} \right).
\end{equation}
We observe that the exponential constant in the global Harnack inequality \eqref{sharpBHI2} is optimal too; see~\cite{SirSou2026a}.

We will also use a slightly more precise version of the interior Harnack inequality
with optimized constants given in  \cite[Theorem 2.1]{SS2}.
To formulate it, we consider a domain $\hat\Omega\supset \Omega$ with $\kappa={\rm dist}(\overline\Omega,\hat\Omega^{c})>0$, and
for $b_1,b_2\in L^q(\hat\Omega)$, $c\in L^p(\hat\Omega)$ with $q\in(n,\infty]$ and $p=q/2$,
 we define the following weaker analogues of the quantities $M,r_0$ in \eqref{defM}-\eqref{defr0}:
\begin{equation}\label{defMhat}
\hat M=\hat M(\ld,\hat\Omega,\kappa) =
\bigl\|b_1\bigr\|^{\beta_q}_{q,\hat r_0,\hat\Omega}+\bigl\|b_2\bigr\|^{\beta_q}_{q,\hat r_0,\hat\Omega}+\|c\|^{\gamma_p}_{p,\hat r_0,\hat\Omega}, 
\end{equation}
\be{defr0hat}
\hat r_0 =\hat r_0(\ld,\hat\Omega,\kappa)
:= \sup\Bigl\{r>0:\: r\bigl(\kappa^{-1}+\|b_1\|^{\beta_q}_{q,r,\hat\Omega}+ \|b_2\|^{\beta_q}_{q,r,\hat\Omega}+\|c\|^{\gamma_p}_{p,r,\hat\Omega}\bigr)\le 1\Bigr\}.
\ee
Analogously to  \eqref{defMmajorsigma1}, \eqref{relMr0}
we have  
\be{relMhatr0}
\hat r_0^{-1}=\kappa^{-1}+\hat M.
\ee

\begin{prop}\label{sharpharn}
 Let $\Omega_1\subset \Omega_2$ be bounded domains  such that  $\kappa={\rm dist}(\overline\Omega_1,\Omega_2^{c})\in(0,D]$ and $\Omega_1$ has geometric diameter $D>0$.
Assume that
$A(x)\in L^\infty(\Omega_2)$ is a matrix such that $\Lambda I \ge A\ge\lambda I$ in $\Omega_2$ and that
$b_1, b_2 \in L^q(\Omega_2)$, $c, g\in L^{p}(\Omega_2)$, for some $q>n$,  $p=q/2$.
If $u\ge0$ satisfies $\ld u= g$ in $\Omega_2$, then
\begin{equation}\label{sharpHarnack}
\sup_{\Omega_1} u\le e^{C_0(\kappa^{-1}+ \hat M)D}\left( \inf_{\Omega_1}u + D^{2-\frac{n}{p}} \|g\|_{L^p(\Omega_2)}\right),
\end{equation}
where $\hat M=\hat M(\ld,\Omega_2,\kappa)$ is defined in \eqref{defMhat}.
\end{prop}
Proposition \ref{sharpharn} follows from straightforward modifications of the proof of \cite[Theorem 2.1]{SS2}, by adjusting the sizes and centers of the balls used to cover $\Omega_1$ and by using \cite[Proposition 8.2(i)]{SirSou2026a}.

\section{Extensions of the main results and proofs}
\label{SecGenRes}

We first state  extensions of Theorems \ref{thm1intro}, \ref{thm2intro1} and \ref{thm4b} to  general $C^{1,\alpha}$ domains, with constants that are sharp and invariant with respect to dilations also for unbounded coefficients. We will use the  notation introduced in the previous section, where we defined the numbers 
 $D$, $r_\Omega$, $M=M(\ld,\Omega)$ and $M_*=M_*(\ld,\Omega)$ (see
 \eqref{defD0D}, \eqref{def_romega} \eqref{defM} and \eqref{defMstar}), on which the following estimates depend.

The first theorem extends the Lipschitz bound \eqref{sharpC0alphapositeigintro} and the eigenvalue bound \eqref{estimlambda1intro}, and also shows they are valid under weaker regularity and integrability assumptions on the coefficients.

\begin{thm} \label{thm1gen}
 Let $\Omega$ be a bounded domain with $C^{1,\bar\alpha}$ boundary.
Assume $\ld$ is as in \eqref{defdiv}, \eqref {hypbaaralpha} holds, 
$A\in  C^{\alpha}(\Omega)$ satisfies \eqref{hyp1}, and 
$$ b_1, b_2 \in L^q(\Omega),\ c\in L^{q/2}(\Omega),\qquad{\rm div}(b_1)+c\le 0\quad\hbox{ in $\mathcal{D}'(\Omega)$.}$$

(i) For each $f\in L^q(\Omega)$, any solution  $v\in H^1(\Omega)$ of
\be{ineqRR0}
-\mathcal{L}v\le f \; \mbox{ in } \Omega,\qquad
v\le0 \; \mbox{on } \partial \Omega,
\ee
satisfies (as above $d=d(x)=\mathrm{dist}(x,\partial\Omega)$)
\be{ineqG02}
\sup_{\Omega} \frac{v}{d}
\le e^{C_0(r^{-1}_\Omega+M_*)D} D^{1-\frac{n}{q}} \|f^+\|_{L^q(\Omega)}.
\ee
Hence the unique solution $u\in H^1_0(\Omega)$ of \eqref{defdiv}
is such that 
\be{ineqG02u}
\left\|\frac{u}{d}\right\|_{L^\infty(\Omega)}
\le  e^{C_0(r^{-1}_\Omega+ M_*)D} D^{1-\frac{n}{q}}\|f\|_{L^q(\Omega)}.
\ee

 (ii) We have
\be{ineqG03u}
\lambda_1(-\mathcal{L},\Omega)\ge e^{-C_0(r_\Omega^{-1}+M_*)D} D^{-2}.
\ee
\end{thm}

The next theorem extends the global bounds \eqref{sharpC1alphapositeigintro}-\eqref{sharpC1alphapositeigintro2}.

\begin{thm} \label{thm2gen1}
Let $\Omega$ be a bounded domain with $C^{1,\bar\alpha}$ boundary. Assume  \eqref{hyp1}-\eqref{hyp2} and ${\rm div}(b_1)+c\le 0$  in $\mathcal{D}'(\Omega)$.
Then the unique solution $u\in H^1_0(\Omega)$ of \eqref{defdiv}  is in $C^{1,\alpha}(\Omega)$, and 
\begin{eqnarray}
\label{sharpC1gen2}
\|\nabla u\|_{L^\infty(\Omega)}
&\le& C_0 \bigl(M+r_\Omega^{-1}\bigr)D^{2-\frac{n}{q}}e^{C_0(r^{-1}_\Omega+ M_*)D} \|f\|_{L^q(\Omega)}\\ 
\label{sharpC1gen3}
[\nabla u]_{\alpha,\Omega}
&\le& C_0 \bigl(M+r_\Omega^{-1}\bigr)^{1+\alpha} D^{2-\frac{n}{q}}e^{C_0(r^{-1}_\Omega+ M_*)D} \|f\|_{L^q(\Omega)}.
\end{eqnarray}
\end{thm}

\begin{rem} \label{remnoloss}
Since $(M_*+r^{-1}_\Omega) 
D=D/r_*$ by \eqref{relMstarr0}, and in view of the exponential factor in \eqref{ineqG02u}-\eqref{sharpC1gen3} and the inequality
$$\|f\|_{q,r_*,\Omega} \le \|f\|_{L^q(\Omega)}\le  c(n)(D/r_*)^{n/q} \|f\|_{q,r_*,\Omega}$$
(since one can cover $\Omega$ by $c(n)(D/r_*)^n$ balls of radius $r_*$),
we see there is no loss in writing  the global norm of $f$ in \eqref{ineqG02u}-\eqref{sharpC1gen3} above, instead of a uniformly local one.
\end{rem}

Further, we give the extension of the $C^0$-to-$C^{1,\alpha}$ estimate (Theorem \ref{thm2intro1}).
\begin{thm} \label{thm2gen2} Let $\Omega$ be a bounded domain with $C^{1,\bar\alpha}$ boundary and   \eqref{hyp1}-\eqref{hyp2} hold.
 If $u\in H^1_0(\Omega)$ solves \eqref{defdiv} then $u\in C^{1,\alpha}(\Omega)$ and
\begin{eqnarray}
\label{sharpC1gen}
\hfill \|\nabla u\|_{L^\infty(\Omega)}
&\le& C_0\Bigl({ \bigl(M+r_\Omega^{-1}\bigr)} \|u\|_{L^\infty(\Omega)}
+ \bigl(M+r_\Omega^{-1}\bigr)^{\frac{n}{q}-1}\|f\|_{q,r_0,\Omega}\Bigr),\\ 
\label{sharpC1alphagen}
\hfill  [\nabla u]_{\alpha, \Omega}
&\le&  C_0\Bigl({\bigl(M+r_\Omega^{-1}\bigr)}^{1+\alpha} \|u\|_{L^\infty(\Omega)}
+ \bigl(M+r_\Omega^{-1}\bigr)^{\frac{n}{q}-1+\alpha}\|f\|_{q,r_0,\Omega}\Bigr).
\end{eqnarray}
\end{thm}

\begin{rem}\label{deduceBR}
By using \eqref{scaleinvtilde1b} and \eqref{defMmajorsigma1}-\eqref{defMmajorsigma2}, it is straightforward to deduce Theorems \ref{thm1intro}, \ref{thm2intro1} and \ref{thm4b} from Theorems \ref{thm1gen}-\ref{thm2gen2}. Note that, if $\Omega=B_R$, then
$D=2R$, $r_\Omega=c(n)R$ and the quantity $e^{\pm C_0(r_\Omega^{-1}+M_*)D}$
above rewrites as
$e^{\pm C_0(1+M_*R)}$.
\end{rem}

Next, we are going to state bounds which also generalize the classical coercivity condition ${\rm div}(b_1)+c\le 0$ in $\Omega$.  We note that the validity of a gradient estimate as in \eqref{lebres} was conjectured to hold in \cite[Conjecture 4.3]{LeB} only assuming uniqueness for the Dirichlet problem for $\ld$ in $\Omega$ (which follows from the  maximum principle, that is, $\lambda_1=\lambda_1(-\mathcal{L},\Omega)>0$).
 It is easy to see that such a result cannot hold without additional assumptions, since for instance  the solution of 
 $\ld_\lambda u:= \Delta u  +\lambda u = -1$ in $B_1$, $u=0$ on $\partial B_1$ explodes together with its gradient as $\lambda\nearrow \lambda_1(-\Delta)$, while  $\lambda_1(-\ld_\lambda, \Omega) = \lambda_1(-\Delta)- \lambda >0$. We observe that it is possible to give estimates in terms of $1/\lambda_1$; however, as such they would be nonquantitative which is not our focus, so some assumption which ensures 
  a positive lower bound on $\lambda_1$ must be present. One such hypothesis in terms of $\Omega$ is that the operator can be extended to a larger domain in which the maximum principle still holds. This is what we assume in the next theorem; apart from yielding quantifiable estimates, it has the important advantages to be what is necessary for our spectral estimates, and to adapt and yield easily the estimates in  Theorems \ref{thm1gen}-\ref{thm2gen1}.

\begin{thm} \label{thm3gen}
Let $\Omega$ be a bounded domain with $C^{1,\bar\alpha}$ boundary and   \eqref{hyp1}-\eqref{hyp2} hold. 
Let   $\kappa\in(0,D]$, and
 $\hat\Omega\supset \Omega$ be such that  $\kappa={\rm dist}(\overline\Omega,\hat\Omega^{c})$.
 Assume  that the coefficients of $\mathcal{L}$ can be extended to $\hat\Omega$ so that
 \be{hyp1a22}
A\in  L^\infty(\hat\Omega)\; \mbox{ with }\;\Lambda I \ge A\ge\lambda I, \quad
  b_1, b_2 \in L^q(\hat\Omega),\quad c\in L^{q/2}(\hat\Omega),\quad\mbox{and} 
   \ee
\be{hypLambda1b2}
\lambda_1(-\mathcal{L},\hat\Omega)\ge 0.
\ee
 Then (recalling the definition of $M$ in \eqref{defM} and $\hat M$ in \eqref{defMhat}):
\smallskip

(i) for each $f\in L^q(\Omega)$, any solution $v\in H^1(\Omega)$ of \eqref{ineqRR0}
 satisfies
\be{estimzinftysubsb}
\sup_{\Omega} \frac{v}{d}
\le e^{C_0(M+\hat M+r_\Omega^{-1}+\kappa^{-1})D} D^{1-\frac{n}{q}} \|f^+\|_{L^q(\Omega)};
\ee

(ii) the unique solution $u\in H^1_0(\Omega)$ of problem \eqref{defdiv} is in $C^{1,\alpha}(\Omega)$ and 
\begin{equation}\label{sharpC1alphapositeigGen}
\left\|\frac{u}{d}\right\|_{L^\infty(\Omega)}+\left\|\nabla u\right\|_{L^\infty(\Omega)}+D^\alpha [\nabla u]_{\alpha,\Omega}
\le e^{C_0(M+\hat M+r_\Omega^{-1}+\kappa^{-1})D} D^{1-\frac{n}{q}}\|f\|_{L^q(\Omega)}; 
\end{equation}

(iii) we have
\begin{equation}\label{lowerbdgen}
\lambda_1(-\mathcal{L},\Omega)\ge e^{-C_0(M+\hat M+r_\Omega^{-1}+\kappa^{-1})D} D^{-2}.
\end{equation}
\end{thm}

\begin{rem}\label{difference}
In the case $\Omega=B_R$,  $\hat\Omega=B_{(1+\theta)R}$, $\theta\in (0,1]$, we have
$D=2R$, $r_\Omega=c(n)R$ so the quantities $e^{\pm C_0(M+\hat M+r_\Omega^{-1}+\kappa^{-1})D}$
in Theorem~\ref{thm3gen} can be replaced by 
$C_0e^{\pm C_0[(M+\hat M)R+\theta^{-1}]}$. For instance, if $\ld$ is as in \eqref{lebeq}, $\Omega=B_R$, $\hat\Omega=B_{2R}$, $\lambda_1(-\mathcal{L},\hat\Omega)\ge 0$,  the latter exponential becomes $C_0e^{C_0M_1R}$, with $M_1=\|b\|_{L^\infty(B_{2R})}+ \|c\|_{L^\infty(B_{2R})}^{1/2}$. Note the important difference with Theorem~\ref{thm1intro} and Theorem \ref{thm4b} in which the dependence in $\|c\|_\infty$ is either absent or polynomial, while in \eqref{estimzinftysubsb}-\eqref{lowerbdgen} it is always exponential. We will show below in Proposition~\ref{Prop-Optim-Linfty2} that this is not due to a loss in the argument,  \eqref{estimzinftysubsb}-\eqref{lowerbdgen} are indeed sharp under the relaxed coercivity assumption \eqref{hypLambda1b2}, even for Schrodinger operators.
\end{rem}

\begin{rem}
Even though our focus in this paper is on the optimal dependence of the constants in the coefficients and the size of the domain, it is also worth recalling  and discussing the much more classical question of the optimality of the Lebesgue norms which appear in the estimates. It is very well known that  $C^1$-estimates fail for coefficients or right-hand sides which are only in $L^n$ instead of $L^q$, $q>n$ (for instance $u(x)=|x|\log|\log|x||$ solves $\Delta u = f\in L^n$, and variations and smooth approximations of that function can be used to violate our estimates for right-hand sides in $L^n$).

On the other hand, we observe that solutions of \eqref{defdiv} are known to be continuously differentiable under weaker hypotheses on the  coefficients or the domain. For instance, we could assume that $A,b_1$ have Dini mean oscillation or the local charts which define the boundary are Dini continuous (see \cite{DEK, DJV24, DJ25}). Similarly, the lower-order coefficients could be in Kato, Lorentz, or Orlicz spaces which are intermediate
 between $L^n$ and $L^q$, $q>n$. Our results should have natural extensions to these cases, however the form of the quantities $M$ and $r_\Omega$ would become rather complicated to write explicitly. Thus, for simplicity, and since our results are new even for operators with smooth coefficients and smooth domains, we refrain from pursuing maximal generality.
\end{rem}

We now give the proofs of the above theorems. We start with Theorem  \ref{thm3gen}.

\begin{proof}[Proof of Theorem \ref{thm3gen}]
(i) The proof relies on optimized interior and boundary Harnack inequalities
 combined with a duality argument.
 We will use the adjoint operator of $\mathcal{L}$, whose expression is
\begin{equation}\label{adj}
\ld^*u= \mathrm{div}(A^T(x)\nabla u -  b_2(x)u) - b_1(x)\cdot\nabla u +c(x) u. 
\end{equation}
Set $\lambda_1:=\lambda_1(-\ld^*,\hat\Omega)=\lambda_1(-\ld,\hat\Omega)$.
Recall the notation (cf.~\eqref{defM}-\eqref{defr0} and
\eqref{defMhat}-\eqref{defr0hat}): 
$$ M =M(\ld,\Omega),\quad r_0 =r_0(\ld,\Omega),\quad
  \hat M =\hat M(\ld,\hat\Omega,\kappa),\quad \hat r_0 =\hat r_0(\ld,\hat\Omega,\kappa).$$
  We also set
 $$  \hat M'  =\hat M(\ld+\lambda_1,\hat\Omega,\kappa),\quad  \hat r_0' =\hat r_0(\ld+\lambda_1,\hat\Omega,\kappa).$$
Note that $\hat M(\ld,\hat\Omega,\kappa)=\hat M(\ld^*,\hat\Omega,\kappa)$, and similarly we can replace $\ld$ by $\ld^*$ in $\hat r_0$, $\hat M'$, $\hat r_0'$. We have $\lambda_1\ge 0$ by assumption and
there exists a first eigenfunction $\phi_1\in H^1_0({ \hat\Omega})$, $\phi_1>0$, satisfying:
\be{supersolPhib}
\ld^*_1\phi_1:=(\ld^*+\lambda_1)\phi_1=0, \quad x\in { \hat\Omega}.
\ee
By standard H\"older regularity (see \cite[Theorem 8.29]{GT}, and also the remark at the end of \cite[Section 8.10]{GT}) we know that $\phi_1\in C(\overline\Omega)$. By the interior Harnack inequality
with optimized constants in Proposition~\ref{sharpharn}, we have
\be{supersolPhi2b}
\sup_{\Omega} \phi_1 \le e^{C_0(\kappa^{-1}+\hat M')D} \inf_{\Omega}\phi_1.
\ee
Since $\lambda_1(-\ld,\Omega)>\lambda_1(-\ld,{ \hat\Omega})\ge 0$, the problem
\be{solwR}
\left\{\hskip 2mm\begin{aligned}
-\mathcal{L}w&=f^+, &\quad&x\in \Omega,\\
w&=0, &\quad&x\in \partial\Omega
\end{aligned}
\right.
\ee
has a unique solution, and it satisfies $w \ge 0$.
By  \cite[Theorem 8.33]{GT}, \cite[Theorem 5.5.5']{Mo}, in view of assumptions \eqref{hyp1}-\eqref{hyp2},
we have $w\in C^1(\overline\Omega)$.
We may thus integrate by parts, applying Lemma~\ref{weakdiv2} from the appendix
in whose statement $\ld$ is replaced by $\ld^*$, $u=\phi_1$ and $v=w$,
to obtain
$$\begin{aligned}
\int_{\Omega} \phi_1 f^+
&= -\int_{\Omega} \phi_1\mathcal{L}w
=- \int_{\Omega} w \mathcal{L}^*\phi_1 -\int_{\partial\Omega}\nu\cdot A(x)\phi_1\nabla w\, d\sigma\\
&=\lambda_1 \int_{\Omega} w\phi_1 -\int_{\partial\Omega}\nu\cdot A(x)\phi_1\nabla w\, d\sigma
\ge -\int_{\partial\Omega}\nu\cdot A(x)\phi_1\nabla w\, d\sigma.
\end{aligned}$$
Observe that $\nabla w=(\partial_\nu w)\nu$, $\partial_\nu w\le 0$ on $\partial\Omega$, and $\nu\cdot A(x)\nu\ge \lambda$.
Consequently,
$$ \lambda \inf_{\Omega} \phi_1  \int_{\partial\Omega} |\partial_\nu w|\, d\sigma
\le \|f^+\|_{L^1(\Omega)}\sup_{\Omega} \phi_1.$$
Using \eqref{supersolPhi2b}, we deduce that
\be{estimwnu0}
\inf_{\partial\Omega} |\partial_\nu w|\le \frac{1}{|\partial\Omega|} \int_{\partial\Omega} |\partial_\nu w|\, d\sigma
\le  \frac{ \|f^+\|_{L^1(\Omega)}}{\lambda|\partial\Omega|} \frac{\sup_{\Omega} \phi_1}{\inf_{\Omega} \phi_1}
\le e^{C_0(\kappa^{-1}+\hat M')D}\frac{\|f^+\|_{L^1(\Omega)}}{|\partial\Omega|}.
\ee
Now, by the sharp boundary Harnack inequality \eqref{sharpBHI} we have
\be{estimwnu}
\sup_{\Omega} \frac{w}{d}\le
e^{C_0(r^{-1}_\Omega+ M)D}\left( \inf_{\Omega} \frac{w}{d} +D^{1-\frac{n}{q}} \|f^+\|_{L^q(\Omega)}\right).
\ee
Since $t^{-1}w(x_0-t\nu)\to |\partial_\nu w(x_0)|$ as $t\to 0^+$ for each $x_0\in \partial\Omega$,
it follows from \eqref{estimwnu0}, \eqref{estimwnu} that
\be{estimwd0}
\begin{aligned}
\sup_{\Omega} \frac{w}{d}
&\le e^{C_0(r^{-1}_\Omega+ M)D}\left(\inf_{\partial\Omega} |\partial_\nu w|+D^{1-\frac{n}{q}} \|f^+\|_{L^q(\Omega)}\right)\\
&\le e^{C_0(r^{-1}_\Omega+  M)D}
\left(e^{C_0(\kappa^{-1}+\hat M')D}\frac{\|f^+\|_{L^1(\Omega)}}{|\partial\Omega|} +D^{1-\frac{n}{q}} \|f^+\|_{L^q(\Omega)}\right).
\end{aligned}
\ee

We now proceed to simplify the quantity on the right hand side of \eqref{estimwd0}.
We first claim that
\be{estimhatM}
 \hat M'
\le C(M+\hat M+r_\Omega^{-1}+\kappa^{-1}).
\ee
If $\hat r_0'\ge \min(r_0,\hat r_0)$, then
$\hat M'= { (\hat r_0') }^{-1}-\kappa^{-1}\le \max(r_0^{-1},\hat r_0^{-1})\le M+r_\Omega^{-1}+\hat M+\kappa^{-1}$
by \eqref{relMr0} and \eqref{relMhatr0}.
We may thus assume
\be{estimhatr1}
\hat r_0'\le \min(r_0,\hat r_0).
\ee
Owing to the definition of $r_\Omega$,
$\Omega$~contains some ball $B$ of radius $r=c(n)r_\Omega$, hence
$M(\mathcal{L},B)\le CM(\mathcal{L},\Omega)=CM$ by Proposition~\ref{scaleinvtilde0}
 with $\omega=B$ and \eqref{scaleinvtilde1b}.
It then follows from Proposition~\ref{upperboundlambda1} and \eqref{relMr0} that
$$\lambda_1=\lambda_1(-\mathcal{L},{ \hat\Omega})\le \lambda_1(-\mathcal{L},B)
\le C\bigl(M(\mathcal{L},B)+r^{-1}\bigr)^2
\le C\bigl(M+r_\Omega^{-1}\bigr)^2=Cr_0^{-2}.$$
Using \eqref{estimhatr1}, $1/\gamma_q=2-n/q$ and \eqref{relMr0} again, we deduce that
$$
\begin{aligned}
 \hat M'
&= \|b_1\|^{\beta_q}_{q,{ \hat r_0'},{ \hat\Omega}}+ \|b_2\|^{\beta_q}_{q,{ \hat r_0'},{ \hat\Omega}}+\|c+\lambda_1\|^{\gamma_q}_{q,{ \hat r_0'},{ \hat\Omega}}\\
&\le  \|b_1\|^{\beta_q}_{q,\hat r_0,{ \hat\Omega}}+ \|b_2\|^{\beta_q}_{q,\hat r_0,{ \hat\Omega}}+C\|c\|^{\gamma_q}_{q,\hat r_0,{ \hat\Omega}}
+C(r_0^{n/q}\lambda_1)^{\gamma_q}
\le C{\hat M}+C\bigl( M+r_\Omega^{-1}\bigr),
\end{aligned}$$
hence \eqref{estimhatM}.
On the other hand, by the definition of $r_\Omega$, it is easy to see that $|\partial\Omega|\ge c(n)r_\Omega^{n-1}$
which, combined with H\"older's inequality and $r_0\le r_\Omega\le D_0\le D$, yields
\be{estimhatM2}
\frac{\|f^+\|_{L^1(\Omega)}}{|\partial\Omega|}
\le C(n)r_\Omega^{1-n}D_0^{n(1-\frac{1}{q})}\|f^+\|_{L^q(\Omega)}
\le C(n)(Dr_0^{-1})^n D^{1-\frac{n}{q}}\|f^+\|_{L^q(\Omega)}.
\ee
Combining \eqref{estimwd0}, \eqref{estimhatM}, \eqref{estimhatM2} and observing that
$$(Dr_0^{-1})^n=(D( M+r^{-1}_\Omega))^n\le e^{C_0( M+r^{-1}_\Omega)D}$$ in view of \eqref{relMr0}, we obtain
\be{estimwnu2}
\sup_{\Omega} \frac{w}{d}
\le e^{C_0(M+\hat M+r^{-1}_\Omega+\kappa^{-1})D} D^{1-\frac{n}{q}} \|f^+\|_{L^q(\Omega)}.
\ee
Finally, for any solution $u\in H^1(\Omega)$ of \eqref{ineqRR0},
since $\lambda_1(-\mathcal{L},\Omega)> \lambda_1(-\mathcal{L},{ \hat\Omega})\ge 0$,
we may apply the maximum principle to get $u\le w$,
hence \eqref{estimwnu2} yields estimate \eqref{estimzinftysubsb}.

\smallskip

(ii) The proof uses Theorem \ref{thm2gen2} so it will be given below after the proof of that theorem. 

\smallskip

(iii) We apply assertion (i) 
with $f=1$. Let $u\in H^1_0(\Omega)$ be the solution of $-\mathcal{L} u=1$ in $\Omega$. By the maximum principle $u$ is positive in $\Omega$. By using that $\|1\|_{L^q(\Omega)}= |\Omega|^{1/q}\le D_0^{n/q}\le D^{n/q}$ in inequality \eqref{sharpC1alphapositeigGen},  we get
$$
-\mathcal{L} u=1\ge e^{-C_0(M+\hat M+r_\Omega^{-1}+\kappa^{-1})D} D^{-2}\,u \quad \mbox{ in }\Omega.
$$
By the characterization \eqref{charlambda1} of the first eigenvalue, the result follows.
\end{proof}

From Theorem \ref{thm3gen} (iii) we deduce Theorem \ref{thmmonot}.

\begin{proof}[Proof of Theorem \ref{thmmonot}]  We will apply Theorem \ref{thm3gen} with $\hat\Omega=B_R$, $\Omega=B_{R-\kappa}$, $\kappa\in (0,R/2]$, $R\ge1$, $D=2R$, $r_\Omega=c(n)R$, $q=\infty$, and $\ld$ replaced by $\hat\ld= \ld -\hat \lambda_1$, where $\hat \lambda_1=\lambda_1(-\mathcal{L},\hat\Omega)$. Then $\lambda_1(-\hat\ld,\hat\Omega)=0$, so \eqref{hyp1a22}-\eqref{hypLambda1b2} hold for $\hat\ld$, hence \eqref{lowerbdgen} is valid with $\ld$ replaced by $\hat\ld$ and $c(x)$ replaced by $c(x) -\hat \lambda_1 $ in $M$, $\hat M$. This yields, by the definition of $M$ and $\hat M$,
\begin{equation}\label{estdifferencedomains}
\lambda_1(-\mathcal{L},\Omega)\ge \hat \lambda_1 
+ e^{-C_0(M+\hat M+|\hat \lambda_1|^{1/2}+R^{-1}+\kappa^{-1})R} R^{-2}.
\end{equation}
Clearly $M=M(\ld,B_{R-\kappa})\le M_1$, $\hat M=\hat M(\ld, B_R)\le M_1$, where $M_1=M(\ld, B_R)$ is the quantity from \eqref{partoper_defM1}, and $M,\hat M$ are the quantities from \eqref{lowerbdgen}  (with $q=p=\infty$).

By Theorem \ref{absupperlambda1} we have $(\hat \lambda_1)_+^{1/2}\le C_0(M_1+R^{-1})$. On the other hand, it is a standard fact (see \cite[Lemma 8.4]{GT}) that the operator $\tilde\ld=\ld-C_0(\|b\|_{L^\infty(B_R)}^2 + \|c\|_{L^\infty(B_R)})$ defines a coercive form over $H^1_0(B_R)$ thus satisfies the maximum principle, that is, $\tilde\ld$ has a positive first eigenvalue in $B_R$. Since that eigenvalue is $\hat\lambda_1 + C_0(\|b\|_{L^\infty(B_R)}^2 + \|c\|_{L^\infty(B_R)})$, we get $(\hat\lambda_1)_-^{1/2}\le C_0(\|b\|_{L^\infty(B_R)}^2 + \|c\|_{L^\infty(B_R)})^{1/2}\le C_0M_1$, and infer \eqref{estimlambda1intromonot} from \eqref{estdifferencedomains}.
\end{proof}

\begin{rem} It is straightforward to check that the proof of Theorem \ref{thmmonot} works for any two domains $\Omega$, $\hat\Omega$ as in Theorem \ref{thm3gen}.
\end{rem}

\begin{proof}[Proof of Theorem \ref{thm1gen}]
(i) We will deduce it from  Theorem~\ref{thm3gen}(i) which we already proved, along with a comparison and extension argument.
By \cite[Theorem~8.1]{GT} or \cite{TR7}, $-\ld$ satisfies the maximum principle.
Replacing $f$ with $f_+$, we may thus assume $f\ge 0$ and $v_{|\partial\Omega}=0$
without loss of generality, hence $v\ge 0$.
Define the operator
\be{defLhat}
\hat\ld w:=\mathrm{div}(A\nabla w)+(b_1+b_2)\cdot \nabla w
\ee
(i.e., $\hat A=A$, $\hat b_1=\hat c=0$, $\hat b_2=b_1+b_2$).
There exists a unique $w\in H^1_0(\Omega)$ such that $-\hat\ld w=f$.
We claim that $v\le w$ in $\Omega$.
Indeed, formally, we have
$$\hat\ld w=-f=\ld v=\mathrm{div}(A\nabla v+b_1 v)+b_2\cdot \nabla v+cv=\hat \ld v+({\rm div}(b_1)+c)v\le \hat\ld v,$$
so the claim should follow from the maximum principle for $-\hat\ld$.
It is standard to make this rigorous, it suffices to argue as in the proof of \cite[Theorem~8.1]{GT} or \cite{TR7},
testing with $\varphi=(v-w-\eps)_+$ and letting $\eps\to 0^+$.

Now, taking $\Omega'=\Omega+B_{r_\Omega}$, hence $\kappa=r_\Omega$, we extend $\hat\ld$ by setting
$\hat A(x)=\lambda I$, $\hat b_1=\hat b_2=\hat c=0$ for $x\in \Omega'\setminus\Omega$.
We will apply Theorem \ref{thm3gen}(i) to the function $w$.
The extended operator $\hat\ld$ satisfies the assumptions of Theorem \ref{thm3gen} (the inequality \eqref{hypLambda1b2}  holds since $\hat\ld 1=0$ in $\Omega'$).
In addition, we have
$$\begin{aligned}
r_0=r_0(\hat\ld,\Omega)
&= \sup\Bigl\{r>0:\: r\bigl(r^{-1}_\Omega+ [\hat A]^{\frac{1}{\alpha}}_{\alpha, r, \Omega} +
\|\hat b_2\|^{\beta_q}_{q,r,\Omega}\bigr)\le 1\Bigr\}\\
&= \sup\Bigl\{r>0:\: r\bigl(r^{-1}_\Omega+ [A]^{\frac{1}{\alpha}}_{\alpha, r, \Omega} +
\|b_1+b_2\|^{\beta_q}_{q,r,\Omega}\bigr)\le 1\Bigr\},
\end{aligned}$$
and we easily check that
$$\begin{aligned}
\hat r_0=\hat r_0(\hat\ld,\Omega',\kappa)
&= \sup\Bigl\{r>0:\: r\bigl(\kappa^{-1}+\|\hat b_2\|^{\beta_q}_{q,r,\Omega'}\bigr)\le 1\Bigr\}\\
&=\sup\Bigl\{r>0:\: r\bigl(r^{-1}_\Omega+\|b_1+b_2\|^{\beta_q}_{q,r,\Omega}\bigr)\le 1\Bigr\}.
\end{aligned}$$
Consequently, $\hat r_0\ge r_0$ and it follows from \eqref{relMr0}, \eqref{relMhatr0} that
$$\kappa^{-1}+\hat M=r^{-1}_\Omega+\hat M=\hat r_0^{-1}\le r_0^{-1}=r^{-1}_\Omega+ M.$$
Inequality \eqref{ineqG02} then readily follows from Theorem \ref{thm3gen}(i)
applied to $w$ and the fact that $v\le w$.
\smallskip

(ii) The proof is similar to that of Theorem \ref{thm3gen}(iii), using  Theorem \ref{thm1gen}(i) instead of  Theorem \ref{thm3gen}(i). 
\end{proof}

\begin{proof}[Proof of Theorem~\ref{thm2gen2}]
It relies on a rescaling argument.
Fix any $x_0\in \overline\Omega$, $r\in (0,r_0]$,
{where $r_0$ is the number from \eqref{defr0},} and let $v(y)=u(x_0+ry)$.
The function $v$ satisfies
$$\tilde{\mathcal{L}}v=\tilde{f}\quad\hbox{ in $\omega:=B_1\cap r^{-1}(\Omega-x_0)$,}$$
where the coefficients of the modified operator $\tilde{\mathcal{L}}$ are $\tilde A(y)= A(x_0+r y)$,
$\tilde b_i(y)=r
b_i(x_0+r y)$, $\tilde c(x)=r^2c(x_0+ry)$, and $\tilde{f}(x)=r^2 f(x_0+ry)$. We compute
\be{rescaledLq}
\begin{aligned}
\|\tilde b_2\|_{L^q(\omega)}
&=r\Bigl(\int_{\omega} |b_2(x_0+ry)|^q\,dy\Bigr)^{1/q}\\
&= r^{1-n/q}\Bigl(\int_{\Omega\cap B_{r}(x_0)} |b_2(x)|^q\,dx\Bigr)^{1/q}
\le r^{1-n/q}\|b_2\|_{q,r_0,\Omega}
\end{aligned}
\ee
and similarly $\|\tilde c\|_{L^p(\omega)} \le r^{2-n/p}\|c\|_{p,r_0,\Omega}$,
$\|\tilde g\|_{L^p(\omega)} \le r^{2-n/p}\|g\|_{p,r_0,\Omega}$. Hence, by the definition of $r_0$,
$$\|\tilde b_2\|_{L^q(\omega)}\le 1,\quad \|\tilde c\|_{L^p(\omega)}\le 1$$
and, similarly,
$$[\tilde A]_{\alpha, 1, \omega}\le 1,\quad \|\tilde b_1\|_{L^\infty(\omega)}\le 1,\quad  [\tilde b_1]_{\alpha, 1, \omega}\le 1.$$

If $x_0\in \partial\Omega$, then we take $r=r_0$ and we have
$v=0$ on $T:=B_1\cap \partial(r^{-1}(\Omega-x_0))$.
Moreover, the local charts defining $T$ are bounded in $C^{1,\alpha}$,
uniformly with respect to $x_0\in \partial \Omega$
(due to $r\le r_0\le r_\Omega$ and the definition of $r_\Omega$ at the beginning of Section~\ref{SecNot}).
By the interior-boundary gradient estimate in \cite[Chapter 8.11]{GT}, \cite[Chapter 5.5]{Mo}, we deduce that
$$\sup_{B_{1/2}\cap  r^{-1}(\Omega-x_0)}
|\nabla v|\le C_0\left(\sup_\omega |v|+ \|\tilde g\|_{L^p(\omega)}\right).$$
Scaling back to $u$ and $f$, we get
$$\sup_{B_{r_0/2}(x_0)\cap \Omega}|\nabla u|\le C_0\left(r_0^{-1}\sup_{B_{r_0}(x_0)\cap \Omega} |u|+
r_0^{1-n/q}\|g\|_{L^p({B_{r_0}(x_0)\cap \Omega})}\right),$$
hence in particular, setting $\Omega'=\{x\in\Omega;\ {\rm dist}(x,\partial\Omega)\le r_0/2\}$,
\begin{equation}\label{EstC1A}
\sup_{\Omega'}|\nabla u|\le C_0\left(r_0^{-1}\sup_{\Omega} |u|
+ r_0^{1-n/q}\|g\|_{p,r_0,\Omega}\right).
\end{equation}

Next, for each $x_0\in \Omega\setminus\Omega'$,
we take $r=r_0/2$, so that $B_r(x_0)\subset \Omega$, hence $\omega=B_1$.
 It follows from the interior gradient estimate in \cite[Chapter 8.11]{GT}, \cite[Chapter 5.5]{Mo}, that
$$\sup_{B_{1/2}} |\nabla v|\le C_0\left(\sup_{B_1} |v|+ \|\tilde g\|_{L^p(B_1)}\right).$$
Scaling back to $u$ and $f$, we deduce that
$$|\nabla u(x_0)|\le 2C_0\left(r_0^{-1}\sup_{B_{r_0}(x_0)} |u|+ r_0^{1-n/q}\|g\|_{L^p({B_{r_0}(x_0)})}\right),$$
hence
\begin{equation}\label{EstC1B}
\sup_{\Omega\setminus\Omega'}|\nabla u|\le C_0\left(r_0^{-1}\sup_{\Omega} |u|
+ r_0^{1-n/q}\|g\|_{p,r_0,\Omega}\right).
\end{equation}
Inequality \eqref{sharpC1gen} follows by combining \eqref{EstC1A}, \eqref{EstC1B} and \eqref{relMr0}.
The proof of \eqref{sharpC1alphagen} is similar.
\end{proof}

 We now prove Theorem~\ref{thm3gen}(ii)
as a consequence of Theorems~\ref{thm3gen}(i) and Theorem \ref{thm2gen2}.

\begin{proof}[Proof of Theorem \ref{thm3gen}(ii)]
Since $\lambda_1(-\mathcal{L},\Omega)> \lambda_1(-\mathcal{L},\Omega')\ge 0$,
there exists a unique solution $u\in H^1_0(\Omega)$ of the equation $-\mathcal{L} u=f$.
It follows from \eqref{estimzinftysubsb} in Theorem \ref{thm3gen}(i) that
\be{estforuh}
\|u\|_{L^\infty(\Omega)}
\le e^{C_0(M+\hat M+r_\Omega^{-1}+\kappa^{-1})D} D^{2-\frac{n}{q}} \|f\|_{L^q(\Omega)},
\ee
since $d(x)\le D$. By combining this inequality with Theorem~\ref{thm2gen2}, we obtain
$$\begin{aligned}
\|\nabla u\|_{L^\infty(\Omega)}
&\le C_0\Bigl({ \bigl(M+r_\Omega^{-1}\bigr)}\|u\|_{L^\infty(\Omega)}
+ \bigl(M+r_\Omega^{-1}\bigr)^{\frac{n}{q}-1}\|f\|_{q,r_0,\Omega}\Bigr)\\
&\le {\bigl(M+r_\Omega^{-1}\bigr)} e^{C_0(M+\hat M+r_\Omega^{-1}+\kappa^{-1})D} D^{2-\frac{n}{q}}\|f\|_{L^q(\Omega)}
+ C_0\bigl(M+r_\Omega^{-1}\bigr)^{\frac{n}{q}-1}\|f\|_{q,r_0,\Omega}\\
&=\Bigl({\bigl(M+r_\Omega^{-1}\bigr)D} \,e^{C_0(M+\hat M+r_\Omega^{-1}+\kappa^{-1})D}
+ C_0\bigl((M+r_\Omega^{-1})D\bigr)^{\frac{n}{q}-1}\Bigr) D^{1-\frac{n}{q}}\|f\|_{L^q(\Omega)}\\
&\le e^{C_0(M+\hat M+r_\Omega^{-1}+\kappa^{-1})D}  D^{1-\frac{n}{q}}\|f\|_{L^q(\Omega)},
\end{aligned}$$
where we used $(M+r_\Omega^{-1})D\ge 1$ and $q>n$ in the last inequality.
The proof of the bound for the H\"older bracket of $\nabla u$ is similar, taking $\alpha\le 1-n/q$ into account.
\end{proof}

\begin{proof}[Proof of Theorem~\ref{thm2gen1}] The proof is similar to that of Theorem \ref{thm3gen}(ii), using
\eqref{ineqG02} instead of \eqref{estimzinftysubsb}.
\end{proof}

Finally, we state and prove a more general version of the quantitative  Landis type estimate  Theorem~\ref{ThmLandisDualityIntro}.
We note 
that in the following result the H\"older regularity assumptions on $b_1, b_2$ are reversed with respect to~\eqref{hyp2}, which  is due to the use of a duality argument.

\begin{thm} \label{ThmLandisDualityB}
Let $R\ge1$ and $\kappa\in(0,R]$.
Assume \eqref{hyp1} with $\Omega=B_{R+\kappa}$, 
$\lambda_1(-\mathcal{L},B_{R+\kappa})\ge 0$ 
 and
\be{hyp2dual}
A, b_2\in C^{\alpha}(B_{R+\kappa}),\quad b_1, c \in L^q(B_{R+\kappa}),
\quad \hbox{ for some $q\in(n,\infty]$.}
\ee
Any nontrivial weak solution  of $\mathcal{L}u=0$ in $B_{R+\kappa}$ satisfies the lower estimate
$$\int_{\partial B_R}\, |u| d\sigma \ge e^{-C_0 (\tilde M+\kappa^{-1})R} R^{-1}\int_{B_R} |u|,$$
where $\tilde M=M(\ld^*,B_{R+\kappa})$.
\end{thm}

\begin{proof}
Let $u$ be a solution of $\mathcal{L}u=0$ in $B_{R+\kappa}$.
By Theorem \ref{thm3gen}, in view of \eqref{hyp2dual} and
since
$\lambda_1(-\mathcal{L}^*,B_R)=\lambda_1(-\mathcal{L},B_R)>\lambda_1(-\mathcal{L},B_{R+\kappa})\ge 0$, the adjoint problem
\be{solvR1}
\left\{\hskip 2mm\begin{aligned}
-\mathcal{L}^*v_R&=sgn(u), &\quad&x\in B_R,\\
v_R&=0, &\quad&x\in S_R
\end{aligned}
\right.
\ee
has a unique solution and it satisfies
\be{solwRestim}
\Bigl\|\frac{v_R}{d}\Bigr\|_\infty\le e^{C_0 (\tilde M+\kappa^{-1})R}\,R^{1-n/q}\|1\|_{L^q(B_R)}
\le Re^{C_0 (\tilde M+\kappa^{-1})R}.
\ee
By \cite[Theorem 8.29]{GT} (see also the remark at the end of \cite[Section 8.10]{GT}) we know that $u\in C(\overline B_R)$.
Also, by \cite[Chapter 8.11]{GT}, \cite[Chapter 5.5]{Mo}, in view of assumption \eqref{hyp2}, we have $v_R\in C^1(\overline B_R)$.
We may thus apply Lemma~\ref{weakdiv2} with $v=v_R$, to obtain
$$
\int_{B_R} |u| = - \int_{B_R} u\mathcal{L}^*v_R = -\int_{S_R}\nu\cdot A^T(x)u\nabla v_R\, d\sigma
\le \|A\|_\infty \int_{S_R}\, |u| |\nabla v_R|\,d\sigma.
$$
Since
$v_R=0$ on $S_R$, it follows from \eqref{solwRestim} that
$$|\nabla v_R(x_0)|=\lim_{t\to 0^+} t^{-1}|v_R(x_0-t\nu)|\le Re^{C_0 (\tilde M+\kappa^{-1})R},\quad\hbox{ for all $x_0\in S_R$.}$$
Consequently,
$$\int_{S_R}\, |u| d\sigma \ge \|A\|_\infty^{-1} R^{-1}e^{-C_0 (\tilde M+\kappa^{-1})R} \int_{B_R} |u|,$$
which implies the conclusion.
\end{proof}

\section{Extensions and proofs of the optimality results}
\label{SecGenRes4}

We start by observing that the optimality of 
estimates \eqref{ineqG02u}-\eqref{sharpC1gen3} in Theorems \ref{thm1gen}-\ref{thm2gen1}  follows from Propositions \ref{Prop-Optim-Linfty-intro}-\ref{Prop-Optim-Linfty-intro-largecn2}, along with the inequalities 
\be{bulinfty}
\|b\|_{q,r_0,B_R} \le C(n) \|b\|^{1-n/q}_{L^\infty(B_R)},
\quad
\|c\|_{q,r_0,B_R} \le C(n) \|c\|^{1-n/2q}_{L^\infty(B_R)},
\ee
valid for 
 $\mathcal{L}_1=\Delta+b(x)\cdot\nabla +c(x)$ 
with any $b, c\in L^\infty(B_R)$ (see \cite[Proposition 3.4]{SirSou2026a}).

 Indeed, by \eqref{bulinfty} 
 we have 
$$
 C(n)\Big(\|b\|_{L^\infty(B_R)} + \|c\|_{L^\infty(B_R)}^{1/2}\Big) \ge \|b\|_{q,r_0,B_R}^{\beta_q} + \|c\|_{q,r_0,B_R}^{\gamma_q} = M(\ld_1,B_R),
$$ 
while the first part of \eqref{bulinfty} applied to  $\ld_2= \Delta + b(x)\cdot\nabla $ (then $r_0(\ld_2)=r_*(\ld_2)$) implies 
$$
C(n) \|b\|_{L^\infty(B_R)}\ge \|b\|_{q, r_*(\ld_2),B_R}^{\beta_q} = M_*(\ld_2,B_R)=  M_*(\ld_1,B_R),
$$
so the same functions obtained below in the proofs of Propositions \ref{Prop-Optim-Linfty-intro}-\ref{Prop-Optim-Linfty-intro-largecn2} serve as examples to the sharpness of \eqref{ineqG02u}-\eqref{sharpC1gen3}. 

\smallskip

The following lemma will be used in the proof of the optimality of the spectral estimates.

\begin{lem} \label{estimlambda}
Let $R>0$ and $\ld=\Delta+b\cdot\nabla+c$ with $b\in C^\alpha(\bar B_R)$, $c\in L^\infty(B_R)$.
Let $u\in C^2(\overline{B_R})$ be such that $u>0$ in $B_R$ and $u=0$ on $\partial B_R$, 
and assume that $f:=-\ld u$ satisfies $\mathcal{F}:={\{f>0\}} \subset\subset B_R$.
Then 
\be{ratiol1}
\lambda_1(-\ld,B_R)\le \frac{\sup_{B_R} f}{\inf_{\mathcal{F}} u}.
\ee
\end{lem}

\begin{proof}
Let $\varphi_1\in H^1_0(B_R)$, $\varphi_1>0$, be the first eigenfunction of $-\ld^*$ in $B_R$, normalized by $\|\varphi_1\|_\infty=1$.
Since $\lambda_1:=\lambda_1(-\ld,B_R)=\lambda_1(-\ld^*,B_R)$, we have by the duality
$$\lambda_1\int_{\mathcal{F}} u\varphi_1\le\lambda_1\int_{B_R} u\varphi_1=-\int_{B_R} u\ld^*\varphi_1
=-\int_{B_R} \varphi_1\ld u=\int_{B_R} f\varphi_1\le \int_{\mathcal{F}} f\varphi_1,$$
where we used integration by parts 
in the second equality.
It follows that
$$\lambda_1\big(\inf_{\mathcal{F}}u\big) \int_{\mathcal{F}} \varphi_1\le
\big(\sup_{B_R} f\big) \int_{\mathcal{F}} \varphi_1.$$
Since $\inf_{\mathcal{F}}u>0$ owing to $\mathcal{F}\subset\subset B_R$, we deduce \eqref{ratiol1}.
\end{proof}

Proposition \ref{Prop-Optim-Linfty-intro} and Proposition \ref{Optim-spectral1} will be proved jointly. We actually directly prove the more general version which also contains the optimality of the estimates in these propositions for the case of unbounded coefficients, when $M_0$ from the introduction is replaced by the finer quantity $M_*$ (as in \eqref{ineqG02u}-\eqref{ineqG03u}).

\begin{proof}[Proof of Proposition~\ref{Prop-Optim-Linfty-intro} and Proposition \ref{Optim-spectral1}]
We fix some function $\varphi\in C^\infty([0,\infty))$, with $\varphi, \varphi'\ge 0$, $\|\varphi^\prime\|_{C^1(\mathbb{R})}\le 2$, such that 
\be{defphiOptim}
\begin{cases}
\varphi(s)=s,& \hbox{for $s\ge \frac12$}\\
\noalign{\vskip 1mm}
\varphi(s)=0,& \hbox{for $s\in[0,\frac14]$,}
\end{cases}
\ee
and we  set 
$$
\delta(r)=\varphi''+{\varphi'}^2+\ts\frac{n-1}{r}\varphi', \qquad
b(x)=-\frac{x\delta(r)}{r}\quad\mbox{ and }\quad\ld:=\Delta+b\cdot\nabla
$$ 
(hence $\lambda_1(-\ld,\R^n)\ge 0$ owing to $\ld 1=0$). We observe that $\|b\|_{L^\infty(\rn)}= C(n)$.
Fixing $R\ge 1$, we define $$z(x)=z(r)=e^{\varphi(r)},\qquad u(x)=u(r)=z(R)-z(r)\in C^\infty(\rn),$$
and we set $f:=-\ld z\in C^\infty(\rn)$. 
Then $u$ solves $\ld u=f$ and $u=0$ on $\partial B_R$ (hence $f\not\equiv 0$ by uniqueness).
Using $z'=\varphi' z$, $z''=(\varphi''+{\varphi'}^2)z$, $\nabla z=\frac{x}{r}\varphi' z$, we obtain
\be{ld2f}
f=\big[-(\varphi''+{\varphi'}^2)-\ts\frac{n-1}{r}\varphi' + \delta\varphi'\big] z
=\big[-1-\ts\frac{n-1}{r}+\delta\big] e^{r}=0, \quad |x|\ge 1,
\ee
hence $\|f\|_{L^q(B_R)}=C_2>0$ for $R\ge 1$.
Recalling that $M=M(\ld,B_{R})=\|b\|^{\beta_q}_{q,r_0,B_{R}}$,
it follows from 
\eqref{MrhoMR}, \eqref{bulinfty} that 
$$
0<c_1(n,q)= M(\ld,B_{1})\le M(\ld,B_{R})\le \|b\|_{L^\infty(\rn)}^{1-n/q} = C_1(n,q).
$$
To obtain \eqref{optimLinfty2} we observe that  $u(0)= e^R-1$, $|u^\prime(R)| = e^R$, so 
\be{optimLinfty1bu1}
\frac{ \min \left\{ \|\nabla u\|_{L^\infty(\partial B_R)},u(0)/R\right\}}{R^{1-\frac{n}{q}}\|f\|_{L^q(B_R)}}\ge \frac{e^{R}-1}{C_2R^2}\ge
Ce^{R/2}\ge  Ce^{CMR},\quad R\ge 1.
\ee

\smallskip

In order to prove  Proposition~\ref{Optim-spectral1}
 it suffices to show the last inequality in \eqref{optim-spectralineq1}.
To this end we apply Lemma~\ref{estimlambda} with the functions $u, f$ above,
noting that $\{f>0\}\subset B_1$, owing to \eqref{ld2f}, and $u\ge e^R-C$ in $B_1$.
\end{proof}

Next, we give the proofs of   Propositions \ref{Prop-Optim-Linfty-intro-largecn1}
-\ref{Prop-Optim-Linfty-intro-largecn2}.

\begin{proof}[Proof of Propositions \ref{Prop-Optim-Linfty-intro-largecn1}-\ref{Prop-Optim-Linfty-intro-largecn2}] 
 Let $\varphi$ be the solution of the initial value problem 
$$\varphi''+\textstyle\frac{n-1}{r}\varphi'=\varphi,\ \ r>0,\qquad \varphi'(0)=0,\quad \varphi(0)=1.$$
It is easy to see that the function $\varphi$ satisfies $\varphi, \varphi', \varphi''>0$ for $r>0$ and $\tilde \varphi(x) = \varphi(r)$ solves $\Delta\tilde \varphi=\tilde \varphi$ in $\R^n$. 

 We first assume $n\ge 3$ and $L>1$, and set
$$\eps=L^{-1},\qquad \rho=L^{-\frac{n-2}{n-1}}$$
 (hence $0<\eps<\rho<1$), and
\be{defLK}
c_n=n(1+\varphi'(1))\mbox{ if } K\ge 1,\quad\mbox{ and }\quad c_n=n(2^n e+1)\varphi'(1)\mbox{ if } K<1.
\ee
We define a function $v(x)=v(r)$ ($r=|x|$), with $v\in C^1([0,\infty))$ piecewise $C^2$, by setting
$$v(r)=
\begin{cases}
e^K\varphi(L r),&0\le r\le \eps\\
\noalign{\vskip 1mm}
v(\eps)+\frac{v'(\eps)}{n-2}\big(\eps -\eps^{n-1}r^{2-n}\big),& \eps<r\le \rho
\end{cases}
$$
and
\be{defvprime}
v'(r)=
\begin{cases}
\rho^{n-1}v'(\rho)r^{1-n}e^{-K(r-\rho)},&\rho<r\le 1+\rho \\
\noalign{\vskip 1mm}
(1+\rho)^{n-1}v'(1+\rho)r^{1-n}-\frac{c_n}{n}(r-(1+\rho)),&1+\rho<r\le 3\\
\noalign{\vskip 1mm}
3^{n-1}v'(3)r^{1-n}e^{K(r-3)},& r>3.
\end{cases}
\ee
Using $\Delta v=r^{1-n}(r^{n-1}v')'$, we obtain
$$\Delta v=
\begin{cases}
L^2v,&0\le r<\eps \\
\noalign{\vskip 1mm}
0,&\eps<r<\rho\\
\noalign{\vskip 1mm}
-Kv',&\rho<r<1+\rho\\
\noalign{\vskip 1mm}
c_n\big(\textstyle\frac{n-1}{n}(1+\rho)r^{-1}-1\big),&1+\rho<r<3\\
\noalign{\vskip 1mm}
Kv',&r>3.
\end{cases}
$$

We claim that \be{optimLinfty2newClaim} 
\hbox{there exists  $R_1=R_1(n)>1$ such that $v(R_1)<0$.} \ee
We observe that $v^\prime>0$ in $(0,\rho)$ and successively compute
$$\begin{aligned}
v(\eps)&=e^K\varphi(1),\qquad v'(\eps)=e^K L \varphi'(1),\\
v(\rho)&\le v(\eps)+\eps\textstyle\frac{v'(\eps)}{n-2}=\frac{n-1}{n-2}e^K \varphi'(1)\le 2e^K \varphi'(1),\\
v'(\rho)&=v'(\eps)\eps^{n-1}\rho^{1-n}=v'(\eps)L^{-1}=e^K \varphi'(1)>0.
\end{aligned}$$
Then, using that $v'>0$ and $v^\prime$ is decreasing on $[\rho,1+\rho]$ by definition,
\be{vrhoeK}
e^K  \varphi(1)<v(1+\rho)\le v(\rho)+v'(\rho)
\le 3e^K \varphi'(1),
\ee
$$
v'(1+\rho)=(1+\rho^{-1})^{1-n}v'(\rho)e^{-K}
\in(0,\varphi'(1)].
$$
Hence $v'$ is also decreasing on $[1+\rho,3]$, so 
$$\begin{aligned}
v(3)&\le v(1+\rho)+2v'(1+\rho)\le 5e^K \varphi'(1),\\
v'(3)&=3^{1-n}(1+\rho)^{n-1}v'(1+\rho)-\textstyle\frac{c_n}{n}(2-\rho)
\le \varphi'(1)-\frac{c_n}{n}.
\end{aligned}$$
By \eqref{defLK} we thus have
\be{majorvprime3}
v'(3)\le 
\begin{cases}
-1,&\hbox{if $K\ge 1$}\\
\noalign{\vskip 1mm}
-2^n e\varphi'(1),& \hbox{otherwise.}
\end{cases}
\ee
Since $v''(r)=3^{n-1}v'(3)(Kr+1-n)r^{-n}e^{K(r-3)}$ for $r>3$, we see  that
\be{vprimedecrease}
\hbox{$v$ is concave on }
\begin{cases}
  [1+\rho,3],&\hbox{ if $K>0$,} \\
  [1+\rho,\infty),&\hbox{ if $K\ge n-1$.}
 \end{cases}
\ee
If $K<1$, it follows that for $r\in[3,6]$
\be{primbd0}
v'(r)\le 3^{n-1}r^{1-n}v'(3)\le 2^{1-n}v'(3)\le -2e\varphi'(1),\ee  hence
$$v(6)\le v(3)-6e\varphi'(1)\le (5e^K-6e)\varphi'(1)<0,$$
and we set $R_1=6$.
If $K\ge1$, we have $v'< 0$ for $r\ge 3$ and 
\be{primbd}
v'\le -3^{n-1}r^{1-n}e^{K(r-3)}\le -C(n)e^{K(r-3)/2}, \quad \mbox{if }\; r\ge 3.
\ee
Consequently, there exists $R_1=R_1(n)>6$ such that, for all $K\ge1$,
$$
\begin{aligned}v(R_1)\le v(3)-C(n) \int_3^{R_1} e^{K(s-3)/2}ds 
&\le 5e^K\varphi'(1)+C(n) K^{-1}(1-e^{K(R_1-3)/2})\\ 
&\le  C(n)K^{-1}e^K(C_1(n)K-e^{K(R_1-5)/2})<0.\end{aligned}$$
This proves claim \eqref{optimLinfty2newClaim}.

Denote by $R_0\in (1+\rho,R_1(n))$ the first zero of $v$.
Setting
$$\hat b=K\textstyle\frac{x}{|x|}\big(1_{\{\rho<|x|<1+\rho\}}-1_{\{3<|x|<R_0\}}\big),
\ \
\hat c=-L^21_{\{|x|<\eps\}},
\ \
\hat f=c_n\big((1+\rho)\textstyle\frac{n-1}{nr}-1\big)1_{\{1+\rho<|x|<3\}}
$$
and using $\nabla v=\textstyle\frac{x}{|x|}v'$, we see that $v\in W^{2,\infty}(B_{R_0})$ is a  strong solution of $\Delta v+\hat b\cdot\nabla v+\hat cv=\hat f$ in $B_{R_0}$, $v>0$ in $B_{R_0}$, $v=0$ on $\partial B_{R_0}$. Moreover,
$$ \hat c, \hat f\le 0,\quad \|\hat c\|_\infty^{1/2}=L,\quad \|\hat b\|_\infty=K,\quad n^{-1}c_n\le \|\hat f\|_\infty\le c_n,
$$
and by using $v(0)=e^K$, $v'(\eps)=e^KL \varphi'(1)$, and  $R_0\le R_1(n)$,
\be{exa1}
\left\|\frac{v}{d}\right\|_{L^\infty(B_{R_0})}\ge \frac{v(0)}{R_0}\, \ge\, C(n) e^K\qquad\mbox{and}\qquad  \|\nabla v\|_{L^\infty(B_{R_0})}\ge v'(\eps)=C(n)L e^K.
\ee

If either $R_0\le 3$, or $R_0\in (3,4)$ and $K\ge n-1$ then, by \eqref{vrhoeK} and \eqref{vprimedecrease}, we get
$v'(R_0)\le \frac{v(R_0)-v(1+\rho)}{R_0-(1+\rho)}\le -\frac13 \varphi(1)e^K =- C(n) e^K$.
 Otherwise we have either $R_0\ge 4$, or $R_0\in (3,4)$ and $K<n-1$. Then $v'(R_0)\le -C(n)e^K$ by \eqref{defvprime}, \eqref{majorvprime3}, and $R_0\le R_1(n)$.
Thus, in both cases, 
\be{exa2}
\|\nabla v\|_{L^\infty(\partial B_{R_0})}\ge C(n)e^K,
\ee
and Propositions \ref{Prop-Optim-Linfty-intro-largecn1}-\ref{Prop-Optim-Linfty-intro-largecn2}
are proved for $R=R_0$ and $L>1$.
  Now fix $L\in[0,1]$. To construct an example for this $L$, let $\hat b$, $\hat c$, $\hat f$, $v$ be the functions obtained above for some $K\ge0$ and  $L=2$, and set $\ld_L:=\Delta +\hat b\cdot\nabla+\frac{L}{2} \hat c$.
 Note that for this operator $\|\frac{L}{2}\hat c\|_\infty=L$.
 Then, by the maximum principle,
 the solution $v_L$ of $\ld_Lv_L=\hat f$ in $B_{R_0}$, $ v_L>0$ in $B_{R_0}$, $v_L=0$ on $\partial B_{R_0}$, is such that $v_L\ge v$.  So \eqref{exa1} and \eqref{exa2} imply
 $$\left\|\frac{v_t}{d}\right\|_{L^\infty(B_{R_0})}\ge \left\|\frac{v}{d}\right\|_{L^\infty(B_{R_0})}\ge C(n) e^K,
 \quad
 \|\nabla v_L\|_{L^\infty(\partial B_{R_0})}\ge \|\nabla v\|_{L^\infty(\partial B_{R_0})}\ge C(n)e^K$$
 hence 
 $\|\nabla  v_L\|_{L^\infty(B_{R_0})}\ge \|\nabla v_L\|_{L^\infty(\partial B_{R_0})}\ge C(n)Le^K,$
   thus settling Propositions~\ref{Prop-Optim-Linfty-intro-largecn1}-\ref{Prop-Optim-Linfty-intro-largecn2} for
   $R=R_0$ and  $L\in [0,1]$ too. 
 For arbitrary $R>0$ the results readily follow  from the already proved  ones for $R_0$,
after rescaling to $B_R$ by setting $\mu=R_0/R$, $u(x)=v(\mu |x|)$, $b(x)=\mu\hat b(\mu |x|)$, 
$c(x)=\mu^2\hat c(\mu|x|)$, $f(x)=\mu^2\hat f(\mu|x|)$.

Finally, if $n=2$, we need only to modify the above proof as follows. We assume $K\ge0$, $L\ge3$, and set 
$$\eps=\frac{1}{L\log(L)},\qquad \rho=\frac{1}{\log(L)},$$
and
$$v(r)=
\begin{cases}
e^K\varphi(L r),&0\le r\le \eps\\
\noalign{\vskip 1mm}
v(\eps)+\eps v^\prime(\eps)(\log(r)-\log(\eps))& \eps<r\le \rho,
\end{cases}
$$
and $v^\prime$ defined in the same way as above for $r>\rho$. 
The rest of the proof is essentially the same, noting also that $\varphi^\prime(t)\sim t$ as $t\to 0$, so $v^\prime(\eps)=e^kL\varphi^\prime((\log L)^{-1})\sim e^K (L/(\log(L))$ for large $L$.
\end{proof}

\begin{rem} A simple computation (left to the reader) shows that the functions constructed above are also optimal for the H\"older bound \eqref{sharpC1alphapositeigintro2}. 
\end{rem}

 It remains to consider the optimality of  the estimates in Theorem~\ref{thm3gen}. Considering our usual model case \eqref{lebeq}, we already saw that under the strong coercivity assumption $c\le0$ in $\Omega$ (for which the trivial extension has positive
first eigenvalue on any larger bounded domain), the estimates in Theorem~\ref{thm3gen}
 are not completely sharp, and can be improved as in Theorems \ref{thm1gen}-\ref{thm2gen1} (whose optimality we just established). In contrast, we will now show that when $c\le0$ in $\Omega$ is relaxed to just positivity of the first eigenvalue of $\ld$, one can do no better than what is stated in Theorem~\ref{thm3gen}, even in the simplest and most classical case of a Schr\"odinger operators (when $b=0$).
 Specifically, restricting to the case $\Omega=B_R$, $\kappa=R$, for simplicity,
and recalling that $M=M(\ld,B_R) =  \|c\|^{\gamma_q}_{q,r_0,B_R} (=\|c\|_{L^\infty(B_R)}^{1/2}$ if $q=\infty$)
 in the Schr\"odinger case, we will show that the optimal constants in the estimates in Theorem~\ref{thm3gen} are indeed exponential in $MR$, either when $M$ or when $R$ becomes large (and in the latter case we can find a fixed function $c$ defined in the whole of $\rn$ which realizes the optimality).

\begin{prop} \label{Prop-Optim-Linfty2}
(i) There exist $c\in C^\infty(\R^n)$ such that, for each $R\ge 1$, 
the operator $\ld:=\Delta+c$ satisfies 
\be{optim-spectralineq1b}
0\le \lambda_1(-\mathcal{L},\rn)<\lambda_1(-\mathcal{L},B_R) \le \tilde C e^{-CRM_R}R^{-2},
\ee
and nontrivial $f\in C^\infty(\R^n)$ such that, for each $q\in(n,\infty]$,
\be{optimLinfty2b}
{\Bigl(\frac{u_{R}}{R}\Bigr)(0)\ge Ce^{CRM_{R}}R^{1-n/q}}\|f\|_{L^q(B_R)}
\qquad\hbox{and}\qquad C_1\le M_R\le C_2,
\ee
where $u_R>0$ is the classical solution of \eqref{defdiv}, $u=0$ on $\partial\Omega$, $M_R:=M(\ld,B_{2R})$ and $C,C_1,C_2>0$ depend only on $n$. 

\smallskip

(ii) For each $R>0$, there exist a sequence $c_j\in C^\infty(\R^n)$ such that
the operator $\ld_{j}:=\Delta+c_j$ satisfies 
\be{optim-spectralineq1c}
0\le \lambda_1(-\mathcal{L}_j,B_{2R})<\lambda_1(-\mathcal{L}_j,B_R) \le \tilde C e^{-CRM_j}R^{-2},
\ee
and a sequence of nontrial $f_j\in C^\infty(\R^n)$
such that, for each $q\in(n,\infty]$,
\be{optimLinfty2c}
\Bigl(\frac{u_{j}}{R}\Bigr)(0)\ge Ce^{CRM_{j}}R^{1-n/q}\|f_{j}\|_{L^q(B_R)}
\qquad\hbox{and}\qquad \hbox{$M_{j}\to \infty$ as $j\to\infty$},
\ee
where $u_j>0$ is the classical solution of \eqref{defdiv}, $u_j=0$ on $\partial\Omega$
with $\ld=\ld_{j}$ and $f=f_{j}$, $M_{j}:=M(\ld_{j},B_{2R})$
and $C>0$ depends only on $n$.
\end{prop}

\begin{rem}\label{remdifsharp}  The examples in Proposition \ref{Prop-Optim-Linfty2} are such that for large $R$ or large $M$ we have $0<\lambda_1(B_R) - \lambda_1(B_{2R})\le  \tilde C e^{-CRM}R^{-2}$, which shows the sharpness of Theorem \ref{thmmonot}. 
\end{rem}

\begin{proof}[Proof of Proposition~\ref{Prop-Optim-Linfty2}](i)
Let $\varphi$ be given by \eqref{defphiOptim}
and  set $c(x)=c(r)= \varphi''-{\varphi'}^2+\ts\frac{n-1}{r}\varphi'$
and $\ld:=\Delta+c$.
Fixing $R\ge 1$ and denoting $r=|x|$, we define $w(x)=w(r)=e^{\varphi(R)-\varphi(r)}\in C^\infty(\rn)$.
Using $w'=-\varphi' w$, $w''=(-\varphi''+{\varphi'}^2)w$, we obtain
$$
\ld w=\big[-\varphi''+{\varphi'}^2-\ts\frac{n-1}{r}\varphi' +c\big]w=0.
$$
Since $w>0$, we deduce in particular that $\lambda_1(-\ld,\rn)\ge 0$.
The function $u(x)=w(x)-1$ is in $C^\infty(\rn)$ and solves $-\ld u=f:=c$, $u=0$ on $\partial B_R$.
Set $M_R:=M(\ld,B_{2R})=\|c\|^{\gamma_q}_{q,r_0,B_{2R}}$.
Since 
\be{cforxge1}
c=\ts\frac{n-1}{r}-1\;\mbox{ if }\; |x|\ge \ts\frac12,\qquad c=0\;\mbox{ if }\; |x|<\frac{1}{4},
\ee 
it follows from 
\eqref{MrhoMR}, \eqref{bulinfty}, that 
\be{optimLinfty1b}
0<c_1(n,q)= M(\ld,B_{1})\le M_R\le \|c\|_{L^\infty(\rn)}^{1-n/(2q)} = C_1(n,q).
\ee
along with $0<\|f\|_{L^q(B_R)}\le C_2R^{\frac{n}{q}}$,
for some constants $C_1,C_2>0$.
Consequently,
\be{optimLinfty1bu0}
\frac{u(0)}{d(0)R^{1-\frac{n}{q}}\|f\|_{L^q(B_R)}}\ge \frac{e^{R}-1}{C(n)R^2}\ge
C(n)e^{R/2}\ge  Ce^{CMR},\quad R\ge 1,
\ee
hence \eqref{optimLinfty2b}.

 To prove the upper bound in \eqref{optim-spectralineq1b}, we apply Lemma~\ref{estimlambda} with the functions $u, f$ above.
Since  $\{f>0\}\subset B_{n-1}$ by  \eqref{cforxge1} and $u\ge ce^R$ on $\bar B_{n-1}$,
inequalities   
\eqref{ratiol1} and \eqref{optimLinfty1b} then yield $\lambda_1\le Ce^{-R}\le C_1e^{-C_0MR}$.

\smallskip

(ii) 
Let $c,f$ be given by assertion (i) and, for each integer $j\ge 1$, let $u_j$ be the corresponding solution of \eqref{defdiv} in $B_j$, $u_j=0$ on $\partial B_j$.
We set 
$$c_j(x)=\ts\frac{j^2}{R^2}c(\frac{jx}{R}),\quad  \tilde f_j(x)=\frac{j^2}{R^2}f(\frac{jx}{R})
\quad\hbox{ and $v_j(x)=u_j(\frac{jx}{R})$ for $|x|\le R$}.$$
Setting also $\hat c_j(x)=4R^2c_j(2Rx)=4j^2c(2jx)$ for $|x|\le 1$, and denoting $\ld_j=\Delta+c_j$, $\hat  \ld_j:=\Delta+\hat c_j$,
it follows from Proposition \ref{scaleinvtilde} that 
\be{2RMj}
2RM(\ld_j,B_{2R})=M(\hat \ld_j,B_1)=2jM(\ld,B_{2j})
\ee
hence, by \eqref{optimLinfty2b},
$$\frac{v_j(0)}{R}=jR^{-1}\frac{u_j(0)}{j}
\ge R^{-1}j^{2-\frac{n}{q}}e^{CjM(\ld,B_{2j})}\|f_j\|_{L^q(B_j)}
=R^{1-\frac{n}{q}}e^{CRM(\ld_j,B_{2R})}\|\tilde f_j\|_{L^q(B_R)}.$$
Moreover, we have $M(\ld_j,B_{2R}) \to\infty$ as $j\to\infty$ by \eqref{optimLinfty1b} and \eqref{2RMj},
and $\lambda_1(-\ld_j,\rn)\ge 0$ owing to $\lambda_1(-\ld,\rn)\ge 0$,
hence \eqref{optimLinfty2c}.

Arguing as for \eqref{2RMj}, we have $RM(\ld_j,B_{R})=jM(\ld,B_{j})$.
Using \eqref{optim-spectralineq1b}, we then obtain
$$\lambda_1(-\ld_j,B_R)=\ts\frac{j^2}{R^2}\lambda_1(-\ld,B_j)\le \ts\frac{j^2}{R^2}\tilde C e^{-CjM(\ld,B_{j})}j^{-2}
=\tilde C e^{-CRM(\ld_j,B_{R})}R^{-2},$$
hence \eqref{optim-spectralineq1c}.
\end{proof}

\begin{rem}\label{gradopt} By direct inspection, we see that either the function $v$ (and its rescalings) which we constructed in the proof of Propositions \ref{Prop-Optim-Linfty-intro-largecn1}-\ref{Prop-Optim-Linfty-intro-largecn2}, or the functions $v_j,f_j$ and the operators $\ld_j$ from 
the proof of Proposition \ref{Prop-Optim-Linfty2} above
 can be used to show the optimality of the $C^0$-to-$C^{1,\alpha}$ estimate in Theorem \ref{thm2intro1} (Theorem \ref{thm2gen2}). This follows  (for $R=1$) from Theorem \ref{thm1intro}
and \eqref{exa1}-\eqref{exa2} for $v$, whereas for $f_j, v_j$ we have
$$\|v_j\|_{L^\infty(B_1)} = v_j(0)=e^j-1,\ \ 
 \|\nabla v_j\|_{L^\infty(B_1)}\sim je^j,\ \ 
M_j\sim j, \ \ 
(M_j+1)^{\frac{n}{q}-1}\|f_j\|_{q,r_j,B_1}\sim j$$
with $M_j=M(\ld_j,B_1)$, $r_j=r_0(\ld_j,B_1)$
(where $a_j\sim b_j$ means $0<C_1\le a_j/b_j\le C_2$).
This shows the optimality with respect to $M\to\infty$ of the gradient estimates \eqref{sharpC1gen2} and \eqref{sharpC1gen} in Theorems \ref{thm2gen1}-\ref{thm2gen2}. A similar computation applies to \eqref{sharpC1gen3} and \eqref{sharpC1alphagen}.
The optimality of these estimates with respect to $R$ when $\Omega=B_R$ follows from their scale invariance.
\end{rem}

\begin{proof}[Proof of Proposition~\ref{PropOptimality}]
We set $v(x)=w(r):=e^{-r}\cos r$ for $r=|x|\ge 1$,
and extend $v$ to a $C^\infty$ function to the whole of $\R^n$, with $v>0$ for $|x|\le 1$.
We also set
$$b_1(x)=\Bigl(2-\frac{n-1}{|x|}\Bigr)\frac{x}{|x|},\quad |x|\ge 1$$
and extend $b_1$ to a $C^\infty$
function on $\R^n$.
We then define
$$c_1(x)=
\begin{cases}
2,&|x|\ge 1\\
\noalign{\vskip 1mm}
v^{-1}\bigl(-\Delta v-b_1\cdot\nabla v\bigr),& |x|<1.
\end{cases}
$$
For $|x|<1$ we have $\mathcal{L}_1v:=\Delta v+b_1\cdot\nabla v+c_1v=0$ by definition.
On the other hand, since $w'(r)=-e^{-r}(\cos r+\sin r)$ and $w''(r)=2e^{-r}\sin r=-2(w'+w)(r)$ for $r>1$, we obtain
$$\Delta v(x)=w''(r)+\frac{n-1}{r}w'(r)=-2w(r)+\Bigl[\frac{n-1}{r}-2\Bigr]w'(r)=
-[b_1\cdot\nabla v+c_1v](x),\quad |x|>1.$$
It follows that $\mathcal{L}_1v(x)=0$ for $|x|>1$ and that the expression in the definition of $c_1(x)$ for $|x|<1$ also
corresponds to $c_1(x)$ for $|x|>1$.
Since $v$ is $C^2$ in $\R^n$, we deduce that $c_1$ is continuous
near $\{|x|=1\}$,
hence in $\R^n$, and that $v$ solves
$\mathcal{L}_1v=0$ in $\R^n$.
\smallskip

Now, since $v(x)=0$ whenever $|x|=(k+\frac12)\pi$ with $k\in\N^*$, the conclusion of  Theorem~\ref{ThmLandisDualityIntro} (which we already proved) fails, hence 
$$\Lambda:=\lambda_1(-\mathcal{L}_1,\R^n)<0.$$
For any $\eps>0$, we set $\tilde b_\eps(x)=\eps b_1(\eps x)$, $\tilde c_\eps(x)=\eps^2 c_1(\eps x)$
and we define the operator $\mathcal{L}_\eps:=\Delta+\tilde b_\eps\cdot\nabla+\tilde c_\eps$.
If for some $\mu\in\R$, $\phi$ is a positive solution of $-\mathcal{L}_1\phi\ge \mu\phi$ in $\R^n$, then
the function $\phi_\eps(x)=\phi(\eps x)$ satisfies
$$[-\mathcal{L}_\eps\phi_\eps](x)=-\eps^2\Delta\phi(\eps x)-\eps b_1(\eps x)\cdot\eps[\nabla\phi](\eps x)
-\eps^2 c_1(\eps x)\phi(\eps x)=-\eps^2[\mathcal{L}_1\phi](\eps x)\ge\eps^2\mu\phi_\eps(x).$$
Consequently, by the properties of first eigenvalue (see the Appendix below),
we have $\lambda_1(-\mathcal{L}_\eps,\R^n)=\eps^2\Lambda$.
But, by the same computation, $u_\eps=v(\eps x)$ is a solution of $\mathcal{L}_\eps u_\eps=0$ in $\R^n$.
Finally choosing $\eps=\sqrt{\lambda/\Lambda}$, the operator $\mathcal{L}=\mathcal{L}_\eps$ and the function
$u=u_\eps$ have all the required properties and the proof of the proposition is complete.
\end{proof}

We close this section with a gradient estimate in the special case $n=1$,
in which the size of the gradient turns out to be  bounded everywhere independently of the zero-order coefficient,
unlike for $n\ge 2$ (cf.~Proposition~\ref{Prop-Optim-Linfty-intro-largecn2} and the paragraph preceding it).

\begin{prop} \label{Prop-Optim-Linfty-intro-largecndim1}
Let $R>0$, $\Omega=(-R,R)$, $a, b_1\in C(\overline\Omega)$,
 $b_2, c,f \in L^1(\Omega)$ with $\lambda:=\inf_\Omega a>0$.
Assume 
\begin{equation}\label{gradn10}
b_1'+c\le 0\quad\hbox{ in $\mathcal{D}'(\Omega)$.}
\ee
Then problem
\eqref{defdiv} has a unique weak solution $u\in C^1(\overline\Omega)$, and $u$ satisfies
\begin{equation}\label{gradn1}
\|u'\|_{L^\infty(\Omega)}\le \lambda^{-1}\,\exp\big(\lambda^{-1}\|b_1+b_2\|_1\big)
\,\|f\|_1.\end{equation}
\end{prop}

\begin{proof} 
The existence-uniqueness part being standard, 
we will only prove \eqref{gradn1}.
Let $b=b_1+b_2$ and $\hat b=a^{-1}b$.
We will more precisely show that, for all $x\in\Omega$,
$$-\int_{-R}^x |f(s)|\exp\Big[\int_x^s \hat b(\tau)d\tau\Big]ds
\le (au')(x)\le \int_x^R 
|f(s)|\exp\Big[\int_x^s \hat b(\tau)d\tau\Big]ds,$$
which immediately implies \eqref{gradn1}.
We first note that, since $au'+b_1u\in C(\overline\Omega)$, the class of test functions can be extended 
by density to
$X:=W^{1,1}_0(\Omega)$; namely, for all $\varphi\in X$,
\be{ext1d}
-\int_{-R}^R (au'+b_1u)\varphi'+\int_{-R}^R (b_2u'+cu-f)\varphi=0.
\ee
Let $v=au'$, $B(x)=\int_{-R}^x \hat b(s)ds$.
First assume $f\le 0$, hence $u\ge 0$ by 
assumption \eqref{gradn10} and the maximum principle.
We claim that
\be{claimD0}
(ve^B)'\ge fe^B\quad\hbox{ in } \mathcal{D}'(\Omega).
\ee
To this end, for $\varphi\in X$, we rewrite \eqref{ext1d} as
$$-\int_{-R}^R au'\varphi'-\int_{-R}^R b_1(u\varphi)'+\int_{-R}^R (bu'+cu-f)\varphi=0.$$
Since $\hat b v=bu'$, we get
$$\int_{-R}^R ve^B(e^{-B}\varphi)'=\int_{-R}^R v(\varphi'-\hat b \varphi)=
\int_{-R}^R (-b_1(u\varphi)'+cu\varphi)-\int_{-R}^R f\varphi.$$
Let $0\le \psi\in \mathcal{D}'(\Omega)$. 
Noting that $B\in W^{1,1}(\Omega)\subset C(\overline\Omega)$, hence $\varphi:=e^{B}\psi\in X$,
and using $0\le u\varphi\in C^1_0(\Omega)$ and assumption  \eqref{gradn10},
 we get
$$\int_{-R}^R ve^B\psi'+\int_{-R}^R fe^{B}\psi
=\int_{-R}^R (-b_1(u\varphi)'+cu\varphi)\le 0,$$
which proves claim \eqref{claimD0}.

Let $w(x):=v(x)e^{B(x)}-\int_{-R}^x f(s)e^{B(s)}ds$.
We have $w\in C(\overline\Omega)$ and
it follows from \eqref{claimD0} that $w$ is nondecreasing.
On the other hand, since $u\ge 0$ and $u(\pm R)=0$, we have
$v(R)\le 0\le v(-R)$. Consequently,
$$0\le v(-R)=w(-R)\le w(x)\le w(R)\le -\int_{-R}^R f(s)e^{B(s)}ds,$$
hence
\be{estimn1}
\int_{-R}^x f(s)e^{B(s)-B(x)}ds\le v(x)\le -\int_x^R f(s)e^{B(s)-B(x)}ds.
\ee

In the general case, write $f=f_++f_-$ and denote by $u_1, u_2\ge 0$
the solution of \eqref{defdiv} with $f$ replaced by $-f_+, f_-\le 0$, respectively.
Then $u=u_2-u_1$ and, applying \eqref{estimn1} to $u_1, u_2$, we get
$$-\int_{-R}^x |f(s)|e^{B(s)-B(x)}ds\le (au')(x)=(a(u'_2-u'_1))(x)\le \int_x^R 
|f(s)|e^{B(s)-B(x)}ds,$$
hence \eqref{gradn1}.
\end{proof}

\section{Appendix} \label{sec-app}

In this appendix we state and/or prove a number of auxiliary or technical results that we have used, and which
were postponed in order not to interrupt the main line of arguments.

\subsection{First eigenvalue}
Let $\Omega$ be a bounded
domain of $\rn$. Using the weak version of the Krein-Rutman theorem (see for instance \cite[Proposition 5.4.32]{DM}), it was proved in \cite{Ch1}, \cite{Ch2}, that any operator $\ld$ satisfying  \eqref{hyp1}, $b_1,b_2\in L^q(\Omega)$, $c,f\in L^{q/2}(\Omega)$, $q>n$,  possesses a first eigenvalue $\lambda_1=\lambda_1(-\ld,\Omega)$ (for the reader's convenience we note that what we call $\lambda_1$ is $-\lambda_1$ in \cite{Ch1}, \cite{Ch2}). This eigenvalue has the usual basic properties, specifically, it is proved in these works that  $\lambda_1$ is a simple eigenvalue, has smallest real part among all eigenvalues, decreases (strictly) with respect to the domain, and corresponds to a positive eigenfunction $\varphi_1\in H^1_0(\Omega)$. We have the characterization
\begin{equation}\label{charlambda1}
\lambda_1=\mathrm{sup}\{\lambda>0\::\: \mbox{there exists } w\in H^1(\Omega),\; w>0,\;  (\ld+\lambda)w\le0 \mbox{ in }\Omega\}.
\end{equation}
Furthermore the validity of the maximum principle for $\ld$ in $\Omega$ is equivalent to the existence of a nonnegative solution $w$ of $\ld[w]<0$ (or nontrivial nonnegative solution of $\ld[w]\le0$) and hence  {\it the positivity of $\lambda_1$ is equivalent to the validity of the maximum principle for $\ld$ in $\Omega$}.  The latter in turn easily implies that $\lambda_1>0$ guarantees the general solvability of the Dirichlet problem for \eqref{defdiv}  (see for instance the beginning of the proof of Proposition 4.1 in \cite{SS2}). We also know that $\varphi_1\in C^\alpha(\Omega)$, by \cite[Theorem 8.29]{GT} and the remark at the end of \cite[Section 8.10]{GT}.

For a more general approach to the existence and properties of the first eigenvalue, see the recent preprint \cite{FGM}. It is also worth observing that under \eqref{hyp1}-\eqref{hyp2} the strong version of the Krein-Rutman theorem (see for instance \cite[Theorem 5.4.33]{DM}) is also applicable, since for sufficiently large $C>0$ the operator $-\ld+C$ is invertible, its inverse is a compact operator from $C^1(\overline{\Omega})$ to itself (by the $C^{1,\alpha}$ estimates in \cite[Section~8.11]{GT}, \cite[Chapter 5.5]{Mo}), as well as strictly monotone on the positive cone of $C^1(\overline{\Omega})$ which has non-empty interior (by the strong maximum principle and the Hopf lemma).

In the case $\Omega=\rn$,
if $\ld$ is an operator whose coefficients are defined on $\rn$ and satisfy the above assumptions
in $B_R$ for every $R>0$, we define
$$ \lambda_1(-\mathcal{L},\R^n):=\lim_{R\to\infty}\lambda_1(-\mathcal{L},B_R)\in[-\infty,\infty).$$

We next state and prove an optimized upper bound on the first eigenvalue, in terms
of the  uniformly local norms of the coefficients of $\ld$ and the domain
(a particular case of which is stated in Theorem \ref{absupperlambda1}), which was used in the proof of some of the main theorems above.

\begin{prop} \label{upperboundlambda1}
Let $R>0$, $\Omega=B_R$ and assume \eqref{hyp1}-\eqref{hyp2}.
Then
\be{upperbdeig1}
\lambda_1(-\ld,B_R) \le C_0\bigl(M+R^{-1}\bigr)^2,
\ee
where $M=M(\ld,B_R)$ (cf. \eqref{defM}).
\end{prop}

The optimality of Proposition~\ref{upperboundlambda1} in terms of $R$ is obvious
(for $R\to\infty$ as well as for $R\to 0$),
since $\lambda_1(-\Delta,B_R) = CR^{-2}$ for the Laplacian.
Its optimality in terms of $M$ as $M\to\infty$ is established by the following proposition.

\begin{prop} \label{Optim-spectral1b} 
For each $R>0$ there exist sequences $b_i, c_i\in C^\infty(\overline{B_{R}})$
such that the operators $\ld_i= \Delta+c_i$ and $\ld_i= \Delta+b_i\cdot\nabla$ satisfy
$$\lambda_1(-\mathcal{L},B_R) \ge  C(M_i+R^{-1})^2, \quad\hbox{where $M_i:=M(\ld_i,B_{R})\to\infty$ as $i\to\infty$.}$$
\end{prop}

\goodbreak

\begin{proof}[Proof of Proposition~\ref{upperboundlambda1}]
{We will show that there exist $C_0, K>0$, depending only on $n,\lambda, \Lambda, \alpha, q$, such that
\be{upperbdeig0}
\begin{aligned}
&\lambda_1:=\lambda_1(-\ld, B_R)\le C_0(R^{-1}+\eta)^2,\ \hbox{ where } \\
&\eta:=[A]^{\frac{1}{\alpha}}_{\alpha,\frac{K}{\sqrt{\lambda_1}},B_R}
+\|b_1\|_{L^\infty(B_R)}+[b_1]^{\frac{1}{1+\alpha}}_{\alpha,\frac{K}{\sqrt{\lambda_1}},B_R}
+\|b_2\|^{\beta_q}_{q,\frac{K}{\sqrt{\lambda_1}},B_R}+\|c\|^{\gamma_q}_{q,\frac{K}{\sqrt{\lambda_1}},B_R}.
\end{aligned}
\ee
We observe that \eqref{upperbdeig0} is equivalent to \eqref{upperbdeig1}.
Indeed, if \eqref{upperbdeig0} holds and $\eta\le M$ then \eqref{upperbdeig1}
is true, whereas
$\eta> M$ implies $K/\sqrt{\lambda_1}\ge r_0:=r_0(\ld,B_R)$ by the definition of $M$, and then $\lambda_1\le K^2r_0^{-2}\le C(R^{-1}+M)^2$ by \eqref{relMr0},
hence again~\eqref{upperbdeig1}.}
Conversely, \eqref{upperbdeig1}
implies $\lambda_1\le Cr_0^{-2}$ by \eqref{relMr0}, i.e. $r_0\le \sqrt{C/\lambda_1}$,
so that $M\le \eta$ with $K= \sqrt{C}$ and \eqref{upperbdeig0} is true.

Let us show \eqref{upperbdeig0}.
By rescaling
\begin{equation}\label{loc11}\tilde A(y)=A(Ry),\quad \tilde b_i(y)=Rb_i(Ry),\quad \tilde c(y)=R^2c(Ry)\end{equation}
we see that $\lambda_1=R^{-2}\tilde\lambda_1$ where $\tilde\lambda_1$ is the eigenvalue of
the corresponding operator $\tilde\ld$ in $B_1$ and we have
$$\|\tilde c\|^{\gamma_q}_{q,K(\tilde\lambda_1)^{-1/2},B_1}=\Bigl\{R^2\|c(R\cdot)\|_{q,\frac{K}{R\sqrt{\lambda_1}},B_1}\Bigr\}^{\gamma_q}
=\Bigl\{R^{2-n/q}\|c\|_{q,\frac{K}{\sqrt{\lambda_1}},B_R}\Bigr\}^{\gamma_q}
=R\|c\|^{\gamma_q}_{q,\frac{K}{\sqrt{\lambda_1}},B_R}$$
\begin{equation}\label{loc12}[\tilde A]^{\frac{1}{\alpha}}_{\alpha,K(\tilde\lambda_1)^{-1/2},B_R}=
\Bigl\{[A(R\cdot)]_{q,\frac{K}{R\sqrt{\lambda_1}},B_1}\Bigr\}^{\frac{1}{\alpha}}
=\Bigl\{R^\alpha[A]_{q,\frac{K}{\sqrt{\lambda_1}},B_R}\Bigr\}^{\frac{1}{\alpha}}
=R[A]^{\frac{1}{\alpha}}_{q,\frac{K}{\sqrt{\lambda_1}},B_R}\end{equation}
and analogous relations for $\tilde b_1, \tilde b_2$.
Consequently it is sufficient to establish \eqref{upperbdeig0} for $R=1$.

Assume for contradiction that \eqref{upperbdeig0} fails. Then, for each integer $j\ge 1$, there exist coefficients $A_j,b_{1,j},b_{2,j},c_j$ such that,
for the corresponding operator $\ld_j$, the eigenvalue $\mu_j:=\lambda_1(-\ld_j, B_1)$ satisfies
\be{mujlarge}
\mu_j\ge j\Bigl(1+
[A]^{\frac{1}{\alpha}}_{\alpha,\frac{j}{\sqrt{\mu_j}},B_1}+
\|b_1\|_{L^\infty(B_1)}+[b_1]^{\frac{1}{1+\alpha}}_{\alpha,\frac{j}{\sqrt{\mu_j}},B_1}
+\|b_{2,j}\|^{\beta_q}_{q,\frac{j}{\sqrt{\mu_j}},B_1 }+ \|c_j\|^{\gamma_q}_{q,\frac{j}{\sqrt{\mu_j}},B_1} \Bigr)^2.
\ee
Let $\varphi_j>0$ be the corresponding eigenfunction normalized by
\be{mujlarge2}
\sup_\Omega \varphi_j= \varphi_j(x_j)=1
\ee
and rescale
$\psi_j(y) = \varphi_j(x)  = \varphi_j(x_j+r_jy)$, where $r_j = 1/\sqrt{\mu_j} \to 0$.
The function $\psi_j$ then satisfies
\be{mujlarge3}
-\tilde\ld_j\psi_j=\psi_j\quad\hbox{in $G_j:=r_j^{-1}(B_1-x_j)$,}
\ee
where the coefficients of $\tilde\ld_j$ are given by
$$\tilde A_j(y)=A(x_j+r_jy),\quad \tilde b_{i,j}(y)=r_jb_{i,j}(x_j+r_jy),\quad \tilde c_j(y)=r_j^2c(x_j+r_jy).$$
Note that $G_j\to G$ where $G$ is either the whole space or a half-space.
For each fixed $L\ge 1$ and large $j$, we have
$$\|\tilde c_j\|_{L^q(B_L\cap G_j)}=r_j^{2-\frac{n}{q}}\|c_j\|_{L^q( B_{Lr_j}(x_j)\cap B_1)}
\le r_j^{2-\frac{n}{q}}\|c_j\|_{q,Lr_j,B_1}=\bigl(\mu_j^{-1}\|c_j\|^{2\gamma_q}_{q,L/\sqrt{\mu_j},B_1}\bigr)^{\frac{1}{2\gamma_q}},$$
$$[\tilde A_j]_{\alpha, B_L\cap  G_j}=r_j^\alpha[A_j]_{\alpha, B_{Lr_j}(x_j) \cap B_1}
\le r_j^\alpha[A_j]_{\alpha,Lr_j,B_1}=\bigl(\mu_j^{-1}[A_j]^{\frac{2}{\alpha}}_{\alpha,L/\sqrt{\mu_j},B_1}\bigr)^{\frac{\alpha}{2}},$$
hence $\|\tilde c_j\|_{L^q(B_L\cap G_j)}, [\tilde A_j]_{\alpha, B_L\cap  G_j}\to 0$ as $j\to\infty$ by \eqref{mujlarge},
and similarly we obtain that $\|b_{1,j}\|_{C^\alpha(B_L\cap  G_j)}\to 0$, $\|b_{2,j}\|_{L^q(B_L\cap G_j)}\to0$.
By \eqref{mujlarge2}, \eqref{mujlarge3}, the $C^{1,\alpha}$ estimates and compact embeddings, up to a subsequence $x_j\to x_0\in \bar B_1$,
we have $A_j\to A^0$ in $C_{loc}^{\alpha/2}(\overline G)$,
where $A^0=A(x_0)$ is a constant matrix satisfying \eqref{hyp1},
and $\psi_j\to \psi^0\ge 0$ in $C^1_{loc}(\overline G)$, where $\psi^0$ satisfies $\psi^0(0)=1$ and
$$
-\ld^0\psi^0= -\mathrm{tr}(A^0D^2\psi^0)= -\mathrm{div}(A^0D\psi^0) = \psi^0\quad \hbox{ in $G$.}
$$
By the standard characterization of the first eigenvalue \eqref{charlambda1} this implies that the first eigenvalue of
$-\ld^0$ is larger or equal to $1$ in any subdomain of $G$.
But $G$ contains balls of arbitrary radius and, since $\ld^0$ has constant coefficients, 
$$
\lambda_1(-\ld^0, B_\rho) = \frac{\lambda_1(-\ld^0, B_1)}{\rho^2}
\to 0 \; \mbox{ as } \; \rho\to \infty,
$$
a contradiction.
\end{proof}

\begin{proof} [Proof of Proposition~\ref{Optim-spectral1b}]
By rescaling $b,c$ as in \eqref{scaledcoeff} and using \eqref{scaleinvtilde5}
and the fact that $\lambda_1(-\mathcal{L},B_R)=R^{-2}\lambda_1(-\tilde{\mathcal{L}},B_1)$, we see that it suffices to consider the case $R=1$.

Let $\lambda\ge 1$. First take $c\equiv-\lambda$ and $\ld=\Delta+c$. By \eqref{bulinfty}, we get
$M:=M(\ld,B_1)
=\|c\|^{\gamma_q}_{q,r_0,B_1}\le C(n)\lambda^{1/2}$, as well as $M\to\infty$ as $\lambda\to\infty$. 
Denoting by $\mu_1$ the first eigenvalue of the Dirichlet Laplacian on $B_1$,
we  have $\lambda_1(-\ld,B_1)=\lambda+\mu_1\ge C_0(M+1)^2$.

Next take $b=-2\lambda e_1$ and $\ld=\Delta+b\cdot\nabla$. By \eqref{bulinfty} we get
$M:=M(\ld,B_1)=\|b\|^{\beta_q}_{q,r_0,B_1}\le C(n)\lambda$, as well as $M\to\infty$ as $\lambda\to\infty$. 
Setting $\phi(x)=e^{\lambda x_1}>0$, we have
$-\ld\phi=(-\lambda^2-\lambda b\cdot e_1)\phi=\lambda^2\phi$, hence
$\lambda_1(-\ld,B_1)\ge \lambda^2\ge C_0(M+1)^2$.
\end{proof}

\subsection{Green's formula.}

In the end we give an integration by parts formula,
which handles situations when one of the two functions
has limited regularity near the boundary. It may be known, but we give a proof since we have not found
a statement suitable to our needs in the literature.

\begin{lem} \label{weakdiv2}
Let $\Omega\subset \rn$ be a bounded $C^1$ domain,
let $\ld$ be an operator as in \eqref{defdiv}, where $A\in C(\overline \Omega)$,
$b_1,c \in L^1(\Omega)$, $b_2 \in L^2(\Omega)$ and let
\begin{equation}\label{Hypweakdiv2}
u\in C(\overline \Omega)\cap H^1(\Omega),\quad v\in C^1(\overline \Omega),\quad\hbox{ with $v=0$ on $\partial \Omega$.}
\end{equation}
Assume that $\mathcal{L}u, \mathcal{L^*}v\in L^1(\Omega)$,
where  $\mathcal{L}u, \mathcal{L^*}v$ are understood in the sense of distributions. Then
$$\int_{\Omega}\bigl\{v\mathcal{L}u-u\mathcal{L}^*v\bigr\}\, dx
=-\int_{\partial\Omega}\nu\cdot A^T(x)u\nabla v\, d\sigma.$$
\end{lem}

\begin{proof} By \cite{Liebe}, there exists a regularized distance, namely a function $\rho\in C^1(\overline\Omega)$
such that, for some constants $C_1,C_2,\eps_0>0$, there holds  $C_1d\le\rho\le C_2d$ in $\Omega$
and $C_1\le |\nabla\rho|\le C_2$ in $\{x\in\overline{\Omega}, d(x)\le\eps_0\}$.
For $\eps\in(0,\eps_0]$, we set
$$G_\eps=\{x\in\Omega,\ \rho(x)>\eps\},\quad \Gamma_\eps=\{x\in\Omega,\ \rho(x)=\eps\}.$$
By the implicit function theorem,
$G_\eps$ is a $C^1$-smooth open set for each $\eps\in(0,\eps_0]$.
We denote the outer normal vector field of $\partial G_\eps$ by $\nu_\eps$
and the surface measure on~$\Gamma_\eps$ by $d\sigma_\eps$.
We also have
$\nu_\eps\,d\sigma_\eps\rightharpoonup \nu\,d\sigma$ as $\eps\to 0$, weakly in the sense of measures, that is
\be{CvSurfMeas}
\lim_{\eps\to  0}\int_{\Gamma_\eps} V\cdot\nu_\eps\,d\sigma_\eps=\int_{\partial\Omega} V\cdot\nu\,d\sigma,
\quad V\in C(\overline\Omega).
\ee
Indeed, if $V\in C^1(\overline\Omega)$, by Green's formula and dominated convergence. we have
$$\int_{\partial\Omega} V\cdot\nu\,d\sigma-\int_{\Gamma_\eps} V\cdot\nu_\eps\,d\sigma_\eps
=\int_{\Omega\setminus G_\eps} {\rm div}\, V\, dx\to 0,$$
and the general case $V\in C(\overline\Omega)$ then follows by using the density of $C^1(\overline\Omega)$ in $C(\overline\Omega)$,
along with the fact that $|\Gamma_\eps|\le C$ for $\eps\in(0,\eps_0]$.

Let $\psi\in C^\infty_0(\R^n)$ with ${\rm Supp}(\psi)= B_1$ satisfy $0\le\psi\le 1$ and $\int_{B_1}\psi\,dx=1$.
Let $(\psi_j)_{j\ge 1}$ be the corresponding sequence of mollifiers, defined by $\psi_j(x)=j^n\psi(jx)$.
Set
$$U=-A^T(x) u\nabla v\in C(\overline\Omega),\quad \phi=A(x)\nabla u+(b_1+b_2)u \in L^1(\Omega),\quad W=U+v\phi.$$
An easy computation (in the distribution sense) yields
\begin{equation}\label{DefWstar}
{\rm div}\, W=v\mathcal{L}u-u\mathcal{L}^*v \in L^1(\Omega).
\end{equation}
Let $\eps\in(0,\eps_0/3)$. Then, for $j> C_2\eps^{-1}$, $W_j:=W\ast \psi_j$ is well defined and smooth  in $\overline G_\eps$
and the divergence formula yields
$$\int_{G_\eps} {\rm div}\, W_j\, dx=\int_{\Gamma_\eps} W_j\cdot\nu_\eps\,d\sigma_\eps.$$
By assumption \eqref{Hypweakdiv2}, we have $|v|\le Cd(x)$ for some constant $C>0$.
Setting $\phi_j=|\phi|\ast \psi_j$, it follows that
$$\delta_j(\eps)
:=\Bigl|\int_{G_\eps} {\rm div}\, W_j\, dx-\int_{\Gamma_\eps} (U\ast \psi_j)\cdot\nu_\eps\,d\sigma_\eps\Bigr|
=\Bigl|\int_{\Gamma_\eps} \bigl((v\phi)\ast \psi_j\bigr)\cdot\nu_\eps\,d\sigma_\eps\Bigr|
\le  C\eps\int_{\Gamma_\eps} \phi_j \,d\sigma_\eps.$$
Since $\Gamma_\eps$ is the $\eps$-level set of the function $\rho$, by the co-area formula, we have
$$\int_\eps^{2\eps} \left(\int_{\Gamma_t}\phi_j\,d\sigma_t\right)\,dt
=\int _{G_\eps\setminus G_{2\eps}} \phi_j|\nabla \rho(x)|\,dx
 \le C_2\int _{G_\eps\setminus G_{2\eps}} \phi_j\,dx.$$
Also, for all $j>C_2\eps^{-1}$, by Fubini's theorem, we have
$$
\int _{G_\eps\setminus G_{2\eps}} \phi_j\,dx
=\int _{G_\eps\setminus G_{2\eps}}\int_{|y|<1/j} |\phi(x-y)| \psi_j(y) \,dy\,dx
\le \int _{\Omega\setminus G_{3\eps}} |\phi(x)| \,dx.
$$
Combining the last three formulas, we obtain
\be{Limweakdiv0}
\frac{1}{2\eps}\int_\eps^{2\eps} \delta_j(r)\,dr
\le CC_2\int_\eps^{2\eps} \left(\int_{\Gamma_t}\phi_j\,d\sigma_t\right)\,dt
\le CC_2\int _{\Omega\setminus G_{3\eps}} |\phi(x)| \,dx,\quad j> C_2\eps^{-1}.
\ee

Now, by \eqref{DefWstar}, we have
$$
\hbox{${\rm div}\,W_j= ({\rm div}\,W)\ast\psi_j\to {\rm div}\,W$ in $L_{loc}^1(\Omega)$},\quad j\to\infty.
$$
Since $U\in C(\overline \Omega)$, we also have
$$
\hbox{$U\ast \psi_j\to U$ in $C_{loc}(\Omega)$}, \quad j\to\infty.
$$
Setting
$\delta_\infty(r):=\bigl|\int_{\Omega_r} {\rm div}\, W\, dx-\int_{\Gamma_r} U\cdot\nu_r\,d\sigma_r\bigr|$,
it follows that $\lim_{j\to\infty}\delta_j(r)=\delta_\infty(r)$ uniformly for $r\in(\eps,2\eps)$ hence,
in view of \eqref{Limweakdiv0},
\be{Limweakdiv1}
\frac{1}{\eps}\int_\eps^{2\eps} \delta_\infty(r)\,dr\le 2CC_2\int _{\Omega\setminus G_{3\eps}} |\phi(x)| \,dx.
\ee

Finally, by
\eqref{CvSurfMeas}, we have
$\lim_{\eps\to  0}\int_{\Gamma_\eps} U\cdot\nu_\eps\,d\sigma_\eps=\int_{\partial\Omega} U\cdot\nu\,d\sigma$,
 and since also
$\lim_{\eps\to  0}$ $\int_{G_\eps} {\rm div}\, W\, dx=\int_\Omega {\rm div}\, W\, dx$
by dominated convergence, we deduce
\be{Limweakdiv2}
\lim_{\eps\to  0}\delta_\infty(\eps)=\int_\Omega {\rm div}\, W\, dx-\int_{\partial\Omega} U\cdot\nu\,d\sigma.
\ee
Passing to the limit $\eps\to 0$ in \eqref{Limweakdiv1}, respectively using \eqref{Limweakdiv2} on the left hand side
and dominated convergence on the right hand side, we conclude that
$$\int_\Omega {\rm div}\, W\, dx-\int_{\partial\Omega} U\cdot\nu\,d\sigma=0,$$
which is what we wanted to prove.
\end{proof}

{\bf Acknowledgements.}
This work was partially done during visits of the second
 author to the Mathematics Departement of the
Pontificia Universidade Cat\'olica do Rio de Janeiro.
He thanks PUC-Rio for the hospitality
and also gratefully acknowledges financial support from the R\'eseau Franco-Br\'esilien de Math\'ematiques (RFBM). The first author was supported by grants CNPq 307772/2022-5 and FAPERJ E-26/204.317/2024.

\end{document}